\documentclass{article}

\usepackage[french,english]{babel}
\usepackage[utf8]{inputenc}
\usepackage[T1]{fontenc}
\usepackage{lmodern}
\usepackage{amsmath}
\usepackage{amssymb}
\usepackage{mathrsfs}
\usepackage{amsthm}
\usepackage[shortlabels]{enumitem}
\usepackage{dsfont} % pour la fonction caractéristique, \mathds{1}
\usepackage{eurosym} % pour le signe  : \euro{}
\usepackage{stmaryrd} % pour les intervalles entiers : $\llbracket 1;n\rrbracket$ pour [|1;n|]
\usepackage{eurosym}

\usepackage{psfrag}

\usepackage{titlesec}
\titleclass{\part}{top}
\titleformat{\part}[display]
{\huge\bfseries\centering}{\partname~\thepart}{0pt}{}
\titlespacing*{\part}{0pt}{160pt}{40pt}

\usepackage{tikz}
\usetikzlibrary{arrows, decorations.pathreplacing}

\usepackage{graphicx}
\usepackage{float}

\usepackage{hyperref}  % pour les hyperliens (table des matières notamment)
\hypersetup{                    % parametrage des hyperliens
	colorlinks=true,                % colorise les liens
	breaklinks=true,                % permet les retours à la ligne pour les liens trop longs
	urlcolor= blue,                 % couleur des hyperliens
	linkcolor= black,                % couleur des liens internes aux documents (index, figures, tableaux, equations,...)
	citecolor= magenta,                % couleur des liens vers les references bibliographiques
	pdfstartview = FitH,
	bookmarksopen = true               % marques-page déroulés
}

\usepackage{listings}
\usepackage{xcolor}
\usepackage{framed}

\usepackage{fullpage} % pour utiliser toute la page

\newcommand{\R}{\mathbb{R}}
\newcommand{\N}{\mathbb{N}}
\newcommand{\E}{\mathbb{E}}

\renewcommand{\epsilon}{\varepsilon}

\newcommand{\1}{\mathds{1}}

\newcommand{\mathbbm}[1]{\1}

\title{Martingale Wasserstein inequality for probability measures in the convex order}
\author{B. Jourdain\thanks{CERMICS, Ecole des Ponts, INRIA, Marne-la-Vallée, France. E-mails: benjamin.jourdain@enpc.fr, william.margheriti@enpc.fr - This research benefited from the support of the “Chaire Risques Financiers”, Fondation du Risque.} \and W. Margheriti\footnotemark[1]}
\date{\today}

\theoremstyle{plain}
\newtheorem{prop2}{Proposition}
\newtheorem{lemma}{Lemma}

\makeatletter
\def\th@remark{%
	\thm@headfont{\itshape\bfseries}%
	\normalfont % body font
	\thm@preskip\topsep \divide\thm@preskip\tw@
	\thm@postskip\thm@preskip
}
\makeatother

\theoremstyle{remark}
\newtheorem{rk}{Remark}
\newtheorem{Ex}{Example}
\newtheorem*{def2}{Definition}

\allowdisplaybreaks

\begin{document}
	\maketitle
	
	\begin{abstract} It was shown by the authors that two one-dimensional probability measures in the convex order admit a martingale coupling with respect to which the integral of $\vert x-y\vert$ is smaller than twice their $\mathcal W_1$-distance (Wasserstein distance with index $1$). We showed that replacing $\vert x-y\vert$ and $\mathcal W_1$ respectively with $\vert x-y\vert^\rho$ and $\mathcal W_\rho^\rho$ does not lead to a finite multiplicative constant. We show here that a finite constant is recovered when replacing $\mathcal W_\rho^\rho$ with the product of $\mathcal W_\rho$ times the centred $\rho$-th moment of the second marginal to the power $\rho-1$. Then we study the generalisation of this new martingale Wasserstein inequality to higher dimension.
	\end{abstract}
	
	\bigskip
	

\section{Introduction}

For all $d\in\N^*$, let ${\cal P}(\R^d)$ denote the set of probability measures on $\R^d$ and for $\rho\ge1$, let  ${\cal P}_\rho(\R^d)$ denote the subset of probability measures with finite $\rho$-th moment. For $\mu,\nu\in\mathcal P_\rho(\R^d)$, we define the Wasserstein distance with index $\rho$ by 
\[
\mathcal W_\rho(\mu,\nu)=\left(\inf_{P\in\Pi(\mu,\nu)}\int_{\R^d\times\R^d}\vert x-y\vert^\rho\,P(dx,dy)\right)^{1/\rho},
\]
where $\Pi(\mu,\nu)$ denotes the set of couplings between $\mu$ and $\nu$, that is 
\[
\Pi(\mu,\nu)=\{P\in\mathcal P_\rho(\R^d\times\R^d)\mid\forall A\in\mathcal B(\R^d),\ P(A\times\R^d)=\mu(A)\ \mathrm{and}\ P(\R^d\times A)=\nu(A)\}.\]

Let $\Pi^{\mathrm{M}}(\mu,\nu)$ be the set of martingale couplings between $\mu$ and $\nu$, that is 
% \[
% \Pi^{\mathrm{M}}(\mu,\nu)=\left\{M\in\Pi(\mu,\nu)\mid\text{$\forall h:\R^d\to\R$ measurable and bounded, }\int_{\R^d\times\R^d}h(x)(y-x)\,M(dx,dy)=0\right\}.
% \]
\[
\Pi^{\mathrm{M}}(\mu,\nu)=\left\{M\in\Pi(\mu,\nu)\mid\mu(dx)\text{-a.e.},\ \int_{\R^d}y\,m(x,dy)=x\right\},
\]
where for all $M\in\Pi(\mu,\nu)$, $(m(x,dy))_{x\in\R}$ denotes a regular conditional probability distribution of $M$ with respect to $\mu$. The celebrated Strassen theorem \cite{St65} ensures that if $\mu,\nu\in\mathcal P_1(\R^d)$, then $\Pi^{\mathrm M}(\mu,\nu)\neq\emptyset$ iff $\mu$ and $\nu$ are in the convex order. We recall that two probability measures $\mu,\nu\in\mathcal P_1(\R^d)$ are in the convex order, and denote $\mu\le_{cx}\nu$, if 
\begin{equation}\label{def:convexOrder}
\int_{\R^d}f(x)\,\mu(dx)\le\int_{\R^d}f(y)\,\nu(dy),
\end{equation}
for any convex function $f:\R^d\to\R$. % We denote $\mu<_{cx}\nu$ if $\mu\le_{cx}\nu$ and $\mu\neq\nu$.
For all $\rho\ge1$ and $\mu,\nu\in\mathcal P_\rho(\R^d)$, we define $\mathcal M_\rho(\mu,\nu)$ by
\[
\mathcal M_\rho(\mu,\nu)=\left(\inf_{M\in\Pi^{\mathrm{M}}(\mu,\nu)}\int_{\R^d\times\R^d}\vert x-y\vert^\rho\,M(dx,dy)\right)^{1/\rho}.
\]

Notice that when $\R^d$ is endowed with the Euclidean norm, the martingale property $\int_{\R^d}\langle x,y\rangle\,M(dx,\newline dy)=\int_{\R^d}\vert x\vert^2\,\mu(dx)$ valid for any martingale coupling $M\in\Pi^{\textrm M}(\mu,\nu)$ yields the remarkable property that $\mathcal M_2(\mu,\nu)$ depends only on the marginals, namely
\begin{equation}\label{expressionM2}
\mathcal M_2^2(\mu,\nu)=\int_{\R^d}\vert y\vert^2\,\nu(dy)-\int_{\R^d}\vert x\vert^2\,\mu(dx)=\int_{\R^d}\vert y-c\vert^2\,\nu(dy)-\int_{\R^d}\vert x-c\vert^2\,\mu(dx),
\end{equation}
for each $c\in\R^d$.

It was shown in \cite{JoMa18} that if $\mu$ and $\nu$ are in the convex order and close to each other, then there exists a martingale coupling which expresses this proximity:
\begin{equation}\label{inegaliteRecherchee2}
\forall \mu,\nu\in\mathcal P_1(\R)\mbox{ such that }\mu\le_{cx}\nu,\quad\mathcal M_1(\mu,\nu)\le 2\mathcal W_1(\mu,\nu),
\end{equation}
where the constant $2$ is sharp. We call this inequality which measures the impact of the restriction to martingale couplings in the minimisation problem defining the Wasserstein distance a martingale Wasserstein inequality. % stability inequality since, as observed at the end of this introduction, it implies stability with respect to its second marginal of the Martingale Optimal Transport value function for a continuous and with at most affine growth cost function when this marginal is increased in the convex order.
It was proved by exhibiting for all $\mu,\nu\in\mathcal P_1(\R)$ in the convex order a subset $\mathcal Q$ of two dimensional probability measures on the unit square and a family $(M^Q)_{Q\in\mathcal Q}$ of martingale couplings between $\mu$ and $\nu$ such that for all $Q\in\mathcal Q$, $\int_{\R\times\R}\vert y-x\vert\,M^Q(dx,dy)\le2 \mathcal W_1(\mu,\nu)$. A particular martingale coupling stands out from the latter family: the so called inverse transform martingale coupling. This coupling is explicit in terms of the cumulative distribution functions of the marginal distributions and their left-continuous generalised inverses. It is more explicit than the left-curtain (and right-curtain) coupling introduced by Beiglböck and Juillet \cite{BeJu16} and which under the condition that $\nu$ has no atoms and the set of local maximal values of $F_\nu-F_\mu$ is finite can be explicited according to Henry-Labordère and Touzi \cite{HeTo13} by solving two coupled ordinary differential equations starting from each right-most local maximiser. Many properties of the inverse transform martingale coupling and the family from which it derives are discussed in \cite{JoMa18}. In this paper, we prove a more general martingale Wasserstein inequality:
\begin{equation}\label{newStabilityInequality}
\forall\rho\ge1,\quad\exists C\in\R_+^*,\quad\forall\mu,\nu\in\mathcal P_\rho(\R)\mbox{ such that }\mu\le_{cx}\nu,\quad\mathcal M_\rho^\rho(\mu,\nu)\le C\mathcal W_\rho(\mu,\nu)\sigma_\rho^{\rho-1}(\nu),
\end{equation}
where the centred moment $\sigma_\rho(\eta)$ of order $\rho$ of $\eta\in\mathcal P_\rho(\R^d)$ is defined by
\[\sigma_\rho(\eta)=\min_{c\in\R^d}\left(\int_{\R^d}\vert y-c\vert^\rho\,\eta(dy)\right)^{1/\rho}.
\]

For all $\rho\ge1$, let $C_\rho$ denote the optimal constant $C$ in \eqref{newStabilityInequality}, that is
\begin{equation}\label{defCrho}
C_\rho=\inf\left\{C>0\mid\forall\mu,\nu\in\mathcal P_\rho(\R)\mbox{ such that }\mu\le_{cx}\nu,\ \mathcal M_\rho^\rho(\mu,\nu)\le C\mathcal W_\rho(\mu,\nu)\sigma_\rho^{\rho-1}(\nu)\right\}.
\end{equation}

One readily notices that \eqref{inegaliteRecherchee2} is a particular case of \eqref{newStabilityInequality} for $\rho=1$ and $C=2$. Moreover, since $2$ is sharp for \eqref{inegaliteRecherchee2}, we have $C_1=2$. One can also obtain that $C_2=2$ when $\R^d$ is endowed with the Euclidean norm with simple arguments which hold in general dimension and actually lead us to generalise \eqref{inegaliteRecherchee2} into \eqref{newStabilityInequality}. Indeed, let $\mu,\nu\in\mathcal P_2(\R^d)$ be such that $\mu\le_{cx}\nu$ and $\pi\in\Pi(\mu,\nu)$ be optimal for $\mathcal W_2(\mu,\nu)$. Then by \eqref{expressionM2}, the martingale property and the Cauchy-Schwarz inequality, we have
\begin{align*}
\begin{split}\label{casRhoEgal2}
\mathcal M_2^2(\mu,\nu)&=\int_{\R^d\times\R^d}\left(\vert y-c\vert^2-\vert x-c\vert^2\right)\,\pi(dx,dy)\\
% &=\int_{\R^d\times\R^d}\langle y-x,y+x\rangle\,\pi(dx,dy)-2\int_{\R^d\times\R^d}\langle y-x,c\rangle\,\pi(dx,dy)\\
&=\int_{\R^d\times\R^d}\langle y-x,y-c+x-c\rangle\,\pi(dx,dy)\le\int_{\R^d\times\R^d}\vert y-x\vert\vert y-c+x-c\vert\,\pi(dx,dy)\\
&\le\mathcal W_2(\mu,\nu)\sqrt{\int_{\R^d\times\R^d}\vert y-c+x-c\vert^2\,\pi(dx,dy)}.
\end{split}
\end{align*}
By Jensen's inequality and the definition of the convex order, the integral in the square root is bounded from above by
$$2\int_{\R^d\times\R^d}\left(\vert y-c\vert^2+\vert x-c\vert^2\right)\,\pi(dx,dy)\le 4\int_{\R^d\times\R^d}\vert y-c\vert^2\nu(dy),$$
where the right-hand side is minimal and equal to $4\sigma^2_2(\nu)$ for $c$ equal to the common mean of $\mu$ and $\nu$ so that $\mathcal M_2^2(\mu,\nu)\le2\mathcal W_2(\mu,\nu)\sigma_2(\nu)$ and $C_2\le 2$. Note that this inequality is prefered to the also derived, sharper but more complex $\mathcal M_2^2(\mu,\nu)\le\mathcal W_2(\mu,\nu)\sqrt{2(\sigma^2_2(\nu)+\sigma^2_2(\nu))}$, since in the limit $\mathcal W_2(\mu,\nu)\to 0$ where the martingale Wasserstein inequality is particularly interesting, $\sigma^2_2(\mu)$ goes to $\sigma_2^2(\nu)$.
% Then by \eqref{expressionM2}, the martingale property, the Cauchy-Schwarz inequality, Jensen's inequality and the definition of the convex order, for all $c\in\R^d$, we have
% \begin{align}
% \begin{split}\label{casRhoEgal2}
% \mathcal M_2^2(\mu,\nu)&=\int_{\R^d\times\R^d}\left(\vert y-c\vert^2-\vert x-c\vert^2\right)\,\pi(dx,dy)\\
% % &=\int_{\R^d\times\R^d}\langle y-x,y+x\rangle\,\pi(dx,dy)-2\int_{\R^d\times\R^d}\langle y-x,c\rangle\,\pi(dx,dy)\\
% &=\int_{\R^d\times\R^d}\langle y-x,y-c+x-c\rangle\,\pi(dx,dy)\le\int_{\R^d\times\R^d}\vert y-x\vert\vert y-c+x-c\vert\,\pi(dx,dy)\\
% &\le\sqrt{\int_{\R^d\times\R^d}\vert y-x\vert^2\,\pi(dx,dy)}\sqrt{\int_{\R^d\times\R^d}\vert y-c+x-c\vert^2\,\pi(dx,dy)}\\
% &\le\mathcal W_2(\mu,\nu)\sqrt{2\int_{\R^d\times\R^d}\left(\vert y-c\vert^2+\vert x-c\vert^2\right)\,\pi(dx,dy)}\\
% &\le\mathcal W_2(\mu,\nu)\sqrt{4\int_{\R^d\times\R^d}\vert y-c\vert^2\,\pi(dx,dy)}=2\mathcal W_2(\mu,\nu)\sqrt{\int_{\R^d}\vert y-c\vert^2\,\nu(dy)}.
% \end{split}
% \end{align}

On the other hand, for all $n\in\N^*$, let $\mu_n$ be the centred Gaussian distribution with variance $n^2$. Then we get that $\mathcal M_2(\mu_n,\mu_{n+1})=\sqrt{2n+1}$. It is well known (see for instance Remark 2.19 (ii) Chapter 2 \cite{Vi1}) that for all $\rho\ge1$ and $\mu,\nu\in\mathcal P_\rho(\R)$,
\begin{equation}
\mathcal W_\rho(\mu,\nu)=\left(\int_0^1\left\vert F_\mu^{-1}(u)-F_\nu^{-1}(u)\right\vert^\rho\,du\right)^{1/\rho},
\end{equation}
where we denote by $F_\eta(x)=\eta((-\infty,x]),x\in\R$ and $F_\eta^{-1}(u)=\inf\{x\in\R\mid F_\eta(x)\ge u\},u\in(0,1)$, the cumulative distribution function and the quantile function of a probability measure $\eta$ on $\R$. Therefore, for $G\sim\mathcal N_1(0,1)$, $\mathcal W_2(\mu_n,\mu_{n+1})=(\int_0^1\vert nF_{\mu_1}^{-1}(u)-(n+1)F_{\mu_{1}}^{-1}(u)\vert^2\,du)^{1/2}=\E[\vert G\vert^2]^{1/2}=1$. We deduce that for all $n\in\N^*$, $2n+1\le C_2\sqrt{(n+1)^2}$, which implies for $n\to+\infty$ that $C_2\ge2$. Hence $C_2=2$.

The generalisation of the martingale Wasserstein inequality \eqref{inegaliteRecherchee2} is motivated by the resolution of the Martingale Optimal Transport (MOT) problem introduced by Beiglböck, Henry-Labordère and Penkner \cite{BeHePe12} in a discrete time setting, and Galichon, Henry-Labordère and Touzi \cite{GaHeTo13} in a continuous time setting. For adaptations of celebrated results on classical
optimal transport theory to the MOT problem, we refer to Beiglböck and Juillet \cite{BeJu16}, Henry-Labordère, Tan and Touzi \cite{HeTato} and Henry-
Labordère and Touzi \cite{GaHeTo13}. On duality, we
refer to Beiglböck, Nutz and Touzi \cite{BeNuTo16}, Beiglböck, Lim and Obłój \cite{BeLiOb} and De March \cite{De18}. We also refer to De March \cite{De18b} and De March and Touzi \cite{DeTo17} for the multi-dimensional case.

About the numerical resolution of the MOT problem, one can look at Alfonsi, Corbetta and
Jourdain \cite{AlCoJo17a, AlCoJo17b}, De March \cite{De18c}, Guo and Obłój \cite{GuOb} and Henry-Labordère \cite{He19}. When $\mu$ and $\nu$ are finitely supported, then the MOT problem amounts to linear programming. In the general case, once the MOT problem is discretised by approximating $\mu$ and $\nu$ by probability measures with finite support and
in the convex order, Alfonsi, Corbetta and Jourdain raised the question of the convergence of the discrete optimal cost towards the
continuous one. Partial results were first brought by Guo and Obłój \cite{GuOb} and the stability of left-curtain couplings obtained by Juillet \cite{Ju14b}. Backhoff-Veraguas and Pammer \cite{BaPa19} and Wiesel \cite{Wi20} independently proved stability of the Martingale Optimal Transport value with respect to the marginal distributions in dimension one under mild regularity assumption on the cost function. Very recently, Br\"uckerhoff and Juillet \cite{BJ} proved that in dimension $d\ge 2$, stability fails and \eqref{newStabilityInequality} does not generalise. Since, on the contrary, for $\rho=2$, the generalisation to any dimension is possible, we may wonder for which values of $\rho$ it is also the case. Note that such a generalization would imply stability of the Martingale Optimal Transport problem for continuous costs which satisfy a growth constraint related to $\rho$ restricted to the case when the second marginal is increased in the convex order. More precisely, let $\rho >1$, $\mu,\nu\in\mathcal P_\rho(\R^d)$ be such that $\mu\le_{cx}\nu$ and $(\nu_n)_{n\in\N}\in\mathcal P_\rho(\R^d)^\N$ be such that $\nu\le_{cx}\nu_n$ for all $n\in\N$ and $\nu_n$ converges to $\nu$ in $\mathcal W_\rho$ as $n\to+\infty$. Let $c:\R^d\times\R^d\to\R$ be continuous and growing at most as the $\rho$-th power of its variables, i.e. $\vert c(x,y)\vert\le K(1+\vert x\vert^\rho+\vert y\vert^\rho)$ for all $(x,y)\in\R^d\times\R^d$ and a certain $K\in\R_+$. It is well known that any sequence $(\pi_n)_{n\in\N}\in\prod_{n\in\N}\Pi^{\mathrm M}(\mu,\nu_n)$ is tight and has all its accumulation points with respect to the weak convergence topology in $\Pi^{\mathrm M}(\mu,\nu)$. Then one can readily derive the first inequality
\[
V(\mu,\nu):=\inf_{\pi\in\Pi^{\mathrm M}(\mu,\nu)}\int_{\R^d\times\R^d}c(x,y)\,\pi(dx,dy)\le\liminf_{n\to+\infty}V(\mu,\nu_n).
\]

On the other hand, for any $(\pi_n)_{n\in\N}\in\prod_{n\in\N}\Pi^{\mathrm M}(\mu,\nu_n)$, we have
\begin{equation}\label{limsupCmunun}
\limsup_{n\to+\infty}V(\mu,\nu_n)\le\limsup_{n\to+\infty}\int_{\R^d\times\R^d}c(x,y)\,\pi_n(dx,dy).
\end{equation}

Recall that a sequence $(\tau_n)_{n\in\N}$ of probability measures on $\R^d\times\R^d$ converges to $\tau$ in $\mathcal W_\rho$ iff the sequence $(\int_{\R^d\times\R^d}f(x,y)\,\tau_n(dx,dy))_{n\in\N}$ converges to $\int_{\R^d\times\R^d}f(x,y)\,\tau(dx,dy)$ for any real-valued continuous map $f:\R^d\times\R^d\to\R$ which grows at most as the $\rho$-th absolute power of its variables. Hence it suffices to find $(\pi_n)_{n\in\N}$ converging in $\mathcal W_\rho$ to some optimal coupling $\pi\in\Pi^{\mathrm M}(\mu,\nu)$ for $V(\mu,\nu)$, which exists by  Lemma \ref{existenceCouplageMartingaleOptimal} below. This lemma also ensures that there exist for all $n\in\N$ a martingale coupling $M_n\in\Pi^{\mathrm M}(\nu,\nu_n)$ optimal for $\mathcal M_\rho(\nu,\nu_n)$. Let $(m_n(y,dy'))_{y\in\R^d}$ be a regular conditional probability distribution of $M_n$ with respect to $\nu$ and $\pi_n(dx',dy')=\int_{y\in\R^d}m_n(y,dy')\,\pi(dx',dy)\in\Pi^{\mathrm M}(\mu,\nu_n)$. Since $\pi(dx,dy)\,\delta_{x}(dx')\,m_n(y,dy')$ is a coupling between $\pi$ and $\pi_n$, we have
\begin{align*}
\mathcal W_\rho^\rho(\pi,\pi_n)&\le\int_{\R^d\times\R^d\times\R^d}\vert y-y'\vert^\rho\,\pi(dx,dy)\,m_n(y,dy')=\int_{\R^d\times\R^d}\vert y-y'\vert^\rho\,M_n(dy,dy')\\
&=\mathcal M_\rho^\rho(\nu,\nu_n)\le C_\rho\mathcal W_\rho(\nu,\nu_n)\sigma_\rho^{\rho-1}(\nu_n).
\end{align*}

If $C_\rho$ is finite, then by convergence of $(\nu_n)_{n\in\N}$ in $\mathcal W_\rho$, the sequences $(\mathcal W_\rho(\nu,\nu_n))_{n\in\N}$ and $(\sigma_\rho^{\rho-1}(\nu_n))_{n\in\N}$ are bounded, hence $(\pi_n)_{n\in\N}$ converges to $\pi$ in $\mathcal W_\rho$. For the above mentionned numerical motivation, this covers the case when the support of $\nu$ is bounded and this measure is approximated by dual quantization.

We present our main result in Section \ref{sec:A new stability inequality}, namely the new one-dimensional martingale Wasserstein inequality which extends the previous one, see \cite{JoMa18}, to any index $\rho\ge1$. Then Section \ref{sec:Towards a multidimensional generalisation} addresses the extension of this inequality to higher dimension. It turns out that the nice example given in \cite{BJ} to prove that the generalisation fails for $\rho=1$ also prevents it for $\rho<\frac{1+\sqrt{5}}{2}$. For $\rho$ larger than this threshold, apart in the particular case $\rho=2$ addressed above, we were not able to prove the generalisation. But we exhibit restricted classes of couples $(\mu,\nu)\in\mathcal P_\rho(\R^d)\times \mathcal P_\rho(\R^d)$ with  $\mu\le_{cx}\nu$ such that the inequality holds with a finite constant $C$ (possibly equal to the one-dimensional constant $C_\rho$) uniform over the class whatever $\rho\ge 1$. In subsection \ref{secext1d}, we exhibit three classes such that $C$ is equal to the one-dimensional constant $C_\rho$. In subsection \ref{secscal}, we deal with the scaling case where $C=3\times 2^{\rho-1}$.

Finally Section \ref{sec:Lemmas} is devoted to the proof of some technical lemmas.

\section{A new martingale Wasserstein inequality in dimension one}
\label{sec:A new stability inequality}

We come back a moment on the family $(M^Q)_{Q\in\mathcal Q}$ parametrised by $\mathcal Q$ mentioned in the introduction since it will have particular significance in the present section. We briefly recall the construction and main properties, see \cite{JoMa18} for an extensive study. Let $\mu,\nu\in\mathcal P_1(\R)$ be such that $\mu\le_{cx}\nu$ and $\mu\neq\nu$. For $u\in[0,1]$ we define
\begin{equation}\label{defPsi+Psi-}
\Psi_+(u)=\int_0^u(F_\mu^{-1}-F_\nu^{-1})^+(v)\,dv\quad\text{and}\quad \Psi_-(u)=\int_0^u(F_\mu^{-1}-F_\nu^{-1})^-(v)\,dv,
\end{equation}
with respective left continuous generalised inverses $\Psi_+^{-1}$ and $\Psi_-^{-1}$. We then define $\mathcal Q$ as the set of probability measures on $(0,1)^2$ with first marginal $\frac{1}{\Psi_+(1)}d\Psi_+$, second marginal $\frac{1}{\Psi_+(1)}d\Psi_-$ and such that $u<v$ for $Q(du,dv)$-almost every $(u,v)\in(0,1)^2$. Since $d\,\Psi_+$ and $d\,\Psi_-$ are concentrated on two disjoint Borel sets, there exists for each $Q\in\mathcal Q$ a probability kernel $(\pi^Q(u,dv))_{u\in(0,1)}$ such that
\begin{align}
&Q(du,dv)=\frac{1}{\Psi_+(1)}d\Psi_+(u)\,\pi^Q(u,dv)=\frac{1}{\Psi_+(1)}d\Psi_-(v)\,\pi^Q(v,du),\label{revQ}\\
|F_\mu^{-1}-F_\nu^{-1}|(u)du&\pi^Q(u,dv)\mbox{ a.e. }
\vert F_\nu^{-1}(v)-F_\mu^{-1}(u)\vert+\vert F_\mu^{-1}(u)-F_\nu^{-1}(u)\vert=\vert F_\nu^{-1}(v)-F_\nu^{-1}(u)\vert.\label{comparaisonFnu-1vFmu-1u}
\end{align}
We define a probability kernel $(\widetilde m^Q(u,dy))_{u\in(0,1)}$ which satisfies for $du$-almost all $u\in(0,1)$ such that $F_\mu^{-1}(u)\neq F_\nu^{-1}(u)$
\begin{align}\label{defmQ}
\begin{split}
&\widetilde m^Q(u,dy)\\
&=\int_{v\in(0,1)}\left(\frac{F_\mu^{-1}(u)-F_\nu^{-1}(u)}{F_\nu^{-1}(v)-F_\nu^{-1}(u)}\delta_{F_\nu^{-1}(v)}(dy)+\frac{F_\nu^{-1}(v)-F_\mu^{-1}(u)}{F_\nu^{-1}(v)-F_\nu^{-1}(u)}\delta_{F_\nu^{-1}(u)}(dy)\right)\,\pi^Q(u,dv),
\end{split}
\end{align}
and $\widetilde m^Q(u,dy)=\delta_{F_\nu^{-1}(u)}(dy)$ for all $u\in(0,1)$ such that $F_\mu^{-1}(u)=F_\nu^{-1}(u)$. Then for $du$-almost all $u\in(0,1)$,
\begin{equation}\label{mtildeQsigneConstant}
\int_\R\vert y- F_\nu^{-1}(u)\vert\,\widetilde m^Q(u,dy)=\vert F_\mu^{-1}(u)-F_\nu^{-1}(u)\vert.
\end{equation}
The measure
\begin{equation}\label{defMQ}
M^Q(dx,dy)=\int_0^1\delta_{F_\mu^{-1}(u)}(dx)\,\widetilde m^Q(u,dy)\,du
\end{equation}
is a martingale coupling between $\mu$ and $\nu$ which satisfies $$\int_{\R\times\R}\vert y-x\vert\,M^Q(dx,dy)\le\int_0^1\int_\R\vert F_\mu^{-1}(u)-F_\nu^{-1}(u)\vert +\vert F_\nu^{-1}(u)-y\vert\,\widetilde m^Q(u,dy)du=2\mathcal W_1(\mu,\nu).$$

We also recall some standard results about cumulative distribution functions and quantile functions since they will prove very handy one-dimensional tools. Proofs can be found for instance in \cite[Appendix A]{JoMa18}. For any probability measure $\eta$ on $\R$:
\begin{enumerate}[(1)]
	\item\label{it:Fcadlag} $F_\eta$, resp. $F_\eta^{-1}$, is right continuous, resp. left continuous, and nondecreasing;
	\item\label{it:inequality equivalence quantile function} For all $(x,u)\in\R\times(0,1)$,
	\begin{equation}\label{eq:equivalence quantile cdf}
	F_\eta^{-1}(u)\le x\iff u\le F_\eta(x),
	\end{equation}
	which implies
	\begin{align}\label{eq:jumps F}
	&F_\eta(x-)<u\le F_\eta(x)\implies x=F_\eta^{-1}(u),\\
	\label{eq:jumps F2}
	\textrm{and}\quad&F_\eta(F_\eta^{-1}(u)-)\le u\le F_\eta(F_\eta^{-1}(u));
	\end{align}
	\item\label{it:Fminus of F} For $\mu(dx)$-almost every $x\in\R$,
	\begin{equation}\label{eq:F-1circF}
	0<F_\eta(x),\quad F_\eta(x-)<1\quad\text{and}\quad F_\eta^{-1}(F_\eta(x))=x;
	\end{equation}
	\item Denoting by $\lambda_{(0,1)}$, resp. $\lambda_{(0,1)^2}$, the Lebesgue measure on $(0,1)$, resp. $(0,1)^2$, we have
	\begin{equation}\label{copule}
	\left((u,v)\mapsto F_\mu(F_\mu^{-1}(u)-)+v\mu(\{F_\mu^{-1}(u)\})\right)_\sharp\lambda_{(0,1)^2}=\lambda_{(0,1)},
	\end{equation}
	where $\sharp$ denotes the pushforward operation.
	\item\label{it:inverse transform sampling} The image of the Lebesgue measure on $(0,1)$ by $F_\eta^{-1}$ is $\eta$.
\end{enumerate}

The property \ref{it:inverse transform sampling} is referred to as inverse transform sampling.

We can now state and prove our main result. For all $\rho\ge1$ and $\mu,\nu\in\mathcal P_\rho(\R)$ in the convex order, we provide an estimate of the martingale Wasserstein function $\mathcal M_\rho(\mu,\nu)$ in terms of the Wasserstein distance $\mathcal W_p(\mu,\nu)$ and the centred $\rho$-th moment of $\nu$.

\begin{prop2}\label{newStabilityInequalityMajoration} Let $\rho\ge1$ and $\mu,\nu\in\mathcal P_\rho(\R)$ be such that $\mu\le_{cx}\nu$. Then
	\begin{enumerate}[(i)]
		\item\label{it:majorationMrhorhoKrho} For all $Q\in\mathcal Q$, the martingale coupling $M^Q\in\Pi^{\mathrm M}(\mu,\nu)$ defined by \eqref{defMQ} satisfies
		\begin{equation}\label{majorationMrhorhoKrho}
		\mathcal M_\rho^\rho(\mu,\nu)\le \int_{\R\times\R}|x-y|^\rho M^Q(dx,dy)\le K_\rho\mathcal W_\rho(\mu,\nu)\sigma_\rho^{\rho-1}(\nu),
		\end{equation}
		where 
		\begin{equation}\label{defKrho}
		K_\rho=\inf\left\{2^{\rho-1}\gamma_1+2(2^{\rho-2}\vee1)\gamma_2\mid(\gamma_1,\gamma_2)\in\R_+^2\text{ and }\forall x\in\R_+,\ \frac{x+x^\rho}{1+x}\le \gamma_1+\gamma_2(1+x)^{\rho-1}\right\}.
		\end{equation}
		
		\item\label{it:encadrementCrho} The constant $C_\rho$ defined by \eqref{defCrho} satisfies $C_1=K_1=2$, $C_\rho=K_\rho=2^{\rho-1}$ when $\rho\ge 2$ and, for $1<\rho<2$,
		\begin{equation}\label{encadrementCrho}\hspace{-10pt}
		2^{\rho-1}\sup_{x\in(1,+\infty)}\frac{x+x^\rho}{(1+x)^\rho}\le C_\rho\le K_\rho\le
		\left(2^{\rho-1}+2\right)\wedge\left(2\sup_{x\in(1,+\infty)}\frac{x+x^{\rho}}{(1+x)^\rho}\right).\end{equation}
		
		\item\label{it:bonsExposants} $\mathcal W_\rho(\mu,\nu)$ and $\sigma_\rho(\nu)$ have the right exponent in \eqref{newStabilityInequality} in the following sense:
		\begin{equation}\label{bonsExposants}
		\forall\rho>1,\quad\forall s\in(1,\rho],\quad\sup_{\substack{\mu,\nu\in\mathcal P_\rho(\R)\\\mu\le_{cx}\nu}}\frac{\mathcal M_\rho^\rho(\mu,\nu)}{\mathcal W_\rho^s(\mu,\nu)\sigma_\rho^{\rho-s}(\nu)}=+\infty.
		\end{equation}
	\end{enumerate}
\end{prop2}
\begin{figure}[ht]
	\begin{center}
		\psfrag{rho}{\large $\rho$}
		\psfrag{majK}{\large upper bound of $K_\rho$ in \eqref{encadrementCrho}}
		\psfrag{K}{\large $K_\rho$}
		\psfrag{minC}{\large lower bound of $C_\rho$ in \eqref{encadrementCrho}}
		\scalebox{0.75}{\includegraphics{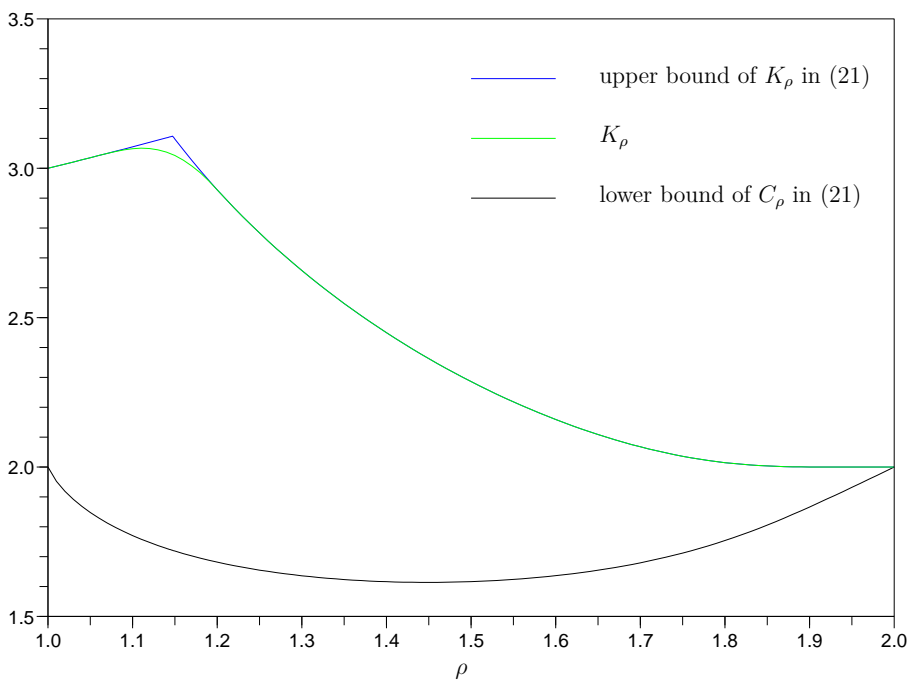}}
		\caption{Plot of $(1,2]\ni\rho\mapsto K_\rho$ with the lower and upper bounds in \eqref{encadrementCrho}.}
		\label{figkrho}\end{center}
\end{figure}

\begin{rk}\label{rknewStabilityInequalityMajoration}
	\begin{enumerate}[(i)]
		\item Let $\rho\in (1,2)$. For $x\ge 0$, $(1+x)^\rho=x^\rho+\int_x^{1+x}\rho y^{\rho-1}dy\le x^\rho+\rho(1+x)^{\rho-1}$. Hence for $x>\rho(1+x)^{\rho-1}$, $x^\rho+\rho(1+x)^{\rho-1}<x^\rho+x$ so that $\sup_{x\in(1,+\infty)}\frac{x+x^\rho}{(1+x)^\rho}>1$. Moreover, $\gamma_2\ge 1$ is necessary for $(\gamma_1,\gamma_2)$ to belong in the set which appears in the definition \eqref{defKrho} of $K_\rho$ and when $\gamma_2=1$, then, by a easy generalization of the previous reasoning, $\gamma_1\ge 1$ is necessary. In the last step of the proof of Proposition \ref{newStabilityInequalityMajoration}, we check that both $(\gamma_1,\gamma_2)=(1,1)$ and $(\gamma_1,\gamma_2)=\left(0,\sup_{x\in(1,+\infty)}\frac{x+x^\rho}{(1+x)^\rho}\right)$ are admissible. The upper bound of $K_\rho$ in \eqref{encadrementCrho} is the minimum of $2^{\rho-1}\gamma_1+2(2^{\rho-2}\vee1)\gamma_2=2^{\rho-1}\gamma_1+2\gamma_2$ over these two couples. In Figure \ref{figkrho}, we plot for $\rho\in (1,2]$, the lower and upper bounds in \eqref{encadrementCrho} together with some numerical estimation of $K_\rho$. The supremum of $x\mapsto f_\rho(x):=\frac{x+x^\rho}{(1+x)^\rho}$ over $(1,+\infty)$ is computed numerically by iterating the function $x\mapsto \frac{\rho x^{\rho-1}+1}{\rho-1}$ in order to find the unique root of the derivative $f'_\rho(x)=\frac{\rho x^{\rho-1}+1-(\rho-1)x}{(1+x)^{\rho+1}}$ and $K_\rho$ is obtained by minimising $2^{\rho-1}\gamma_1+2\sup_{x\in\R_+}\frac{x^\rho+(1-\gamma_1)x-\gamma_1}{(1+x)^\rho}$ over $\gamma_1$ in a grid of the interval $[0,1]$ with the supremum computed in the same way as for $\gamma_1=0$.
		It turns out that away from the neighbourhood $[1.1,1.2]$ of the point where the two functions of $\rho$ involved in the minimum in the right-hand side of \eqref{encadrementCrho} intersect, this right-hand side coincides with $K_\rho$.
		\item When $\rho=1$, the condition $\gamma_2\ge 1$ is no longer necessary for $(\gamma_1,\gamma_2)$ to belong in the set which appears in the definition \eqref{defKrho} of $K_\rho$ and $(\gamma_1,\gamma_2)=(2,0)$ is admissible so that $K_1\le 2$, while $\lim_{\rho\to 1+}K_\rho$ appears to be equal to $3$ according to Figure \ref{figkrho}.
		\item When $\rho\ge 2$, in the last step of the proof of the proposition, we obtain that $K_\rho\le 2^{\rho-1}$ from the admissibility of $(\gamma_1,\gamma_2)=(0,1)$. 
		\item We show (see \eqref{majorationWrhorho} below) that $\mathcal W_\rho(\mu,\nu)\le2\sigma_\rho(\nu)$, so that for $s\in[0,1]$,
		\[
		\mathcal M_\rho^\rho(\mu,\nu)\le2^{1-s}C_\rho\mathcal W_\rho^s(\mu,\nu)\sigma_\rho^{\rho-s}(\nu).
		\]
	\end{enumerate}
\end{rk}

\begin{proof}[Proof of Proposition \ref{newStabilityInequalityMajoration}] Let us prove \ref{it:bonsExposants} first. One can readily show (see for instance \cite[(2.24)]{JoMa18}) that for $a,b\in\R$ such that $0<a<b$,
	\begin{equation}\label{uniqueCouplageMartRademacher2}
	H=\frac{(b+a)}{4b}\delta_{(-a,-b)}+\frac{(b-a)}{4b}\delta_{(-a,b)}+\frac{(b+a)}{4b}\delta_{(a,b)}+\frac{(b-a)}{4b}\delta_{(a,-b)}
	\end{equation}
	is the only martingale coupling between $\mu=\frac12\delta_{-a}+\frac12\delta_a$ and $\nu=\frac12\delta_{-b}+\frac12\delta_b$. Consequenly, for $\rho\ge1$ we trivially have
	\[
	\mathcal M_\rho^\rho(\mu,\nu)=\int_{\R\times\R}\vert x-y\vert^\rho\,H(dx,dy)=\frac{1}{2b}\left((a+b)(b-a)^\rho+(b-a)(a+b)^\rho\right).
	\]
	
	On the other hand, since $\mathcal W_\rho(\mu,\nu)=\left(\int_0^1\vert F_\mu^{-1}(u)-F_\nu^{-1}(u)\vert\,du\right)^{1/\rho}$ (see for instance Remark 2.19 (ii) Chapter 2 \cite{Vi1}),
	\[
	\mathcal W_\rho(\mu,\nu)=\left(\int_0^{1/2}\vert -a-(-b)\vert^\rho\,du+\int_{1/2}^{1}\vert a-b\vert^\rho\,du\right)^{1/\rho}=b-a.
	\]
	
	Moreover, for all $c\in\R$, $\int_\R\vert y-c\vert^\rho\,\nu(dy)=\frac12(\vert b-c\vert^\rho+\vert b+c\vert^\rho)$, which attains its infimum for $c=0$, hence $\sigma_\rho(\nu)=b$. So for all $s\in[1,\rho]$, we have
	\begin{equation}\label{minorationCrho}
	\frac{\mathcal M^\rho_\rho(\mu,\nu)}{\mathcal W_\rho^s(\mu,\nu)\sigma_\rho(\nu)^{\rho-s}}=\frac{1}{2b^{\rho+1-s}}\left((a+b)(b-a)^{\rho-s}+(a+b)^\rho(b-a)^{1-s}\right)\ge\frac{(a+b)^\rho(b-a)^{1-s}}{2b^{\rho+1-s}},
	\end{equation}
	which tends to $+\infty$ as $b$ tends to $a$ as soon as $\rho>1$ and $s\in(1,\rho]$, which proves \ref{it:bonsExposants}. Furthermore, \eqref{minorationCrho} applied with $s=1$, $a=1$ and $b>1$ yields
	\[
	\frac{\mathcal M^\rho_\rho(\mu,\nu)}{\mathcal W_\rho(\mu,\nu)\sigma_\rho(\nu)^{\rho-1}}=\frac{(1+b)(b-1)^{\rho-1}+(1+b)^\rho}{2b^{\rho}}\cdot
	\]
	
	In particular for $b=\frac{x+1}{x-1}$ where $x$ denotes any real number in $(1,+\infty)$, the latter equality writes
	\[
	\frac{\mathcal M^\rho_\rho(\mu,\nu)}{\mathcal W_\rho(\mu,\nu)\sigma_\rho(\nu)^{\rho-1}}=2^{\rho-1}\frac{x+x^\rho}{(1+x)^\rho},
	\]
	which proves the lower bound in \eqref{encadrementCrho} :
	\begin{equation}
	C_\rho\ge 2^{\rho-1}\sup_{x\in(1,+\infty)}\frac{x+x^\rho}{(1+x)^\rho}.\label{minocrho}
	\end{equation} Note that considering more general measures $\mu$ and $\nu$ in the convex order, each concentrated on two atoms, does not yield a greater lower bound.
	
	We now show \ref{it:majorationMrhorhoKrho}. Let $Q\in\mathcal Q$. Since the probability measure $M^Q$ defined by \eqref{defMQ} belongs to $\Pi^{\textrm M}(\mu,\nu)$, we have by definition of $\mathcal M_\rho(\mu,\nu)$ and the definition \eqref{defmQ} of $\widetilde m^Q$ that
	\begin{align}\label{majorationMrhorho1}\begin{split}
	\mathcal M_\rho^\rho(\mu,\nu)&\le\int_{\R\times\R}\vert y-x\vert^\rho\,M^Q(dx,dy)=\int_{(0,1)\times\R}\vert y-F_\mu^{-1}(u)\vert^\rho\,du\,\widetilde m^Q(u,dy)\\
	&=\int_{(0,1)^2}\vert F_\nu^{-1}(v)-F_\mu^{-1}(u)\vert^\rho\frac{\vert F_\mu^{-1}(u)-F_\nu^{-1}(u)\vert}{\vert F_\nu^{-1}(v)-F_\nu^{-1}(u)\vert}\1_{\{\vert F_\mu^{-1}(u)-F_\nu^{-1}(u)\vert>0\}}\,du\,\pi^Q(u,dv)\\
	&+\int_{(0,1)^2}\vert F_\nu^{-1}(u)-F_\mu^{-1}(u)\vert^\rho\frac{\vert F_\nu^{-1}(v)-F_\mu^{-1}(u)\vert}{\vert F_\nu^{-1}(v)-F_\nu^{-1}(u)\vert}\1_{\{\vert F_\mu^{-1}(u)-F_\nu^{-1}(u)\vert>0\}}\,du\,\pi^Q(u,dv).
	\end{split}
	\end{align}
	Let us recall \eqref{comparaisonFnu-1vFmu-1u} :	
	$$|F_\mu^{-1}(u)-F_\nu^{-1}(u)|du\pi^Q(u,dv)\mbox{ a.e. }
	\vert F_\nu^{-1}(v)-F_\mu^{-1}(u)\vert+\vert F_\mu^{-1}(u)-F_\nu^{-1}(u)\vert=\vert F_\nu^{-1}(v)-F_\nu^{-1}(u)\vert.$$
	Let $\gamma_1,\gamma_2\ge0$ be such that for all $x\in\R_+$, $\frac{x+x^\rho}{1+x}\le \gamma_1+\gamma_2(1+x)^{\rho-1}$. Therefore, for all $(a,b)\in\R_+\times\R_+$ such that $a+b>0$, we have
	\begin{equation}\label{inegalitec1c2}
	\frac{a^\rho b+ab^\rho}{a+b}\le\gamma_1a^\rho+\gamma_2a(a+b)^{\rho-1}.
	\end{equation}
	
	By \eqref{majorationMrhorho1}, \eqref{comparaisonFnu-1vFmu-1u} and \eqref{inegalitec1c2} with $(a,b)=(\vert F_\nu^{-1}(u)-F_\mu^{-1}(u)\vert,\vert F_\nu^{-1}(v)-F_\mu^{-1}(u)\vert)$, we get
	\begin{align}\label{majorationMrhorho2}\begin{split}
	\mathcal M_\rho^\rho(\mu,\nu)&\le \gamma_1\int_{(0,1)}\vert F_\nu^{-1}(u)-F_\mu^{-1}(u)\vert^\rho\,du\\
	&\phantom{\le\ }+\gamma_2\int_{(0,1)}\vert F_\nu^{-1}(v)-F_\nu^{-1}(u)\vert^{\rho-1}\vert F_\mu^{-1}(u)-F_\nu^{-1}(u)\vert\,du\,\pi^Q(u,dv)\\
	&=\gamma_1\mathcal W_\rho^\rho(\mu,\nu)+\gamma_2\int_{(0,1)}\vert F_\nu^{-1}(v)-F_\nu^{-1}(u)\vert^{\rho-1}\vert F_\mu^{-1}(u)-F_\nu^{-1}(u)\vert\,du\,\pi^Q(u,dv).
	\end{split}
	\end{align}
	
	Let $c\in\R$. On the one hand, the inverse transform sampling and the definition \eqref{def:convexOrder} of the convex order applied with $x\mapsto\vert x-c\vert^\rho$ yield
	\begin{align}\label{majorationWrhorho}\begin{split}
	&\mathcal W_\rho^\rho(\mu,\nu)\\
	&=\int_{(0,1)}\vert F_\nu^{-1}(u)-F_\mu^{-1}(u)\vert^\rho\,du\le2^{\rho-1}\left(\int_{(0,1)}\vert F_\nu^{-1}(u)-c\vert^\rho\,du+\int_{(0,1)}\vert F_\mu^{-1}(u)-c\vert^\rho\,du\right)\\
	&=2^{\rho-1}\left(\int_\R\vert y-c\vert^\rho\,\nu(dy)+\int_\R\vert x-c\vert^\rho\,\mu(dx)\right)\le2^\rho\int_\R\vert y-c\vert^\rho\,\nu(dy).
	\end{split}
	\end{align}
	
	We deduce that 
	\begin{equation}\label{eq3}
	\mathcal W_\rho^\rho(\mu,\nu)=\mathcal W_\rho(\mu,\nu)\mathcal W_\rho^{\rho-1}(\mu,\nu)\le\mathcal W_\rho(\mu,\nu)\times2^{\rho-1}\left(\int_\R\vert y-c\vert^\rho\,\nu(dy)\right)^{(\rho-1)/\rho}.
	\end{equation}
	
	On the other hand we have
	\begin{align}\label{majorationMrhorho3}\begin{split}
	&\int_{(0,1)}\vert F_\nu^{-1}(v)-F_\nu^{-1}(u)\vert^{\rho-1}\vert F_\mu^{-1}(u)-F_\nu^{-1}(u)\vert\,du\,\pi^Q(u,dv)\\
	&=\int_{(0,1)^2}\vert F_\nu^{-1}(v)-F_\nu^{-1}(u)\vert^{\rho-1}(F_\mu^{-1}-F_\nu^{-1})^+(u)\,du\,\pi^Q(u,dv)\\
	&\phantom{=}+\int_{(0,1)^2}\vert F_\nu^{-1}(v)-F_\nu^{-1}(u)\vert^{\rho-1}(F_\mu^{-1}-F_\nu^{-1})^-(u)\,du\,\pi^Q(u,dv).
	\end{split}
	\end{align}
	
	Using the inequality $\vert x-y\vert^{\rho-1}\le(2^{\rho-2}\vee1)(\vert x\vert^{\rho-1}+\vert y\vert^{\rho-1})$ valid for all $(x,y)\in\R$ and the fact that $(F_\mu^{-1}-F_\nu^{-1})^+(u)\,du\,\pi^Q(u,dv)=Q(du,dv)=(F_\mu^{-1}-F_\nu^{-1})^-(v)\,dv\,\pi^Q(v,du)$ according to \eqref{revQ}, we get
	\begin{align}\label{majorationMrhorho4}\begin{split}
	&\int_{(0,1)^2}\vert F_\nu^{-1}(v)-F_\nu^{-1}(u)\vert^{\rho-1}(F_\mu^{-1}-F_\nu^{-1})^+(u)\,du\,\pi^Q(u,dv)\\
	&\le(2^{\rho-2}\vee1)\left(\int_{(0,1)^2}\vert F_\nu^{-1}(v)-c\vert^{\rho-1}\,(F_\mu^{-1}-F_\nu^{-1})^+(u)\,du\,\pi^Q(u,dv)\right.\\
	&\phantom{\le(2^{\rho-2}\vee1)}\left.+\int_{(0,1)^2}\vert F_\nu^{-1}(u)-c\vert^{\rho-1}\,(F_\mu^{-1}-F_\nu^{-1})^+(u)\,du\,\pi^Q(u,dv)\right)\\
	&=(2^{\rho-2}\vee1)\left(\int_{(0,1)^2}\vert F_\nu^{-1}(u)-c\vert^{\rho-1}(F_\mu^{-1}-F_\nu^{-1})^-(u)\,du\,\pi^Q(u,dv)\right.\\
	&\phantom{\le(2^{\rho-2}\vee1)}\left.+\int_{(0,1)^2}\vert F_\nu^{-1}(u)-c\vert^{\rho-1}(F_\mu^{-1}-F_\nu^{-1})^+(u)\,du\,\pi^Q(u,dv)\right)\\
	&=(2^{\rho-2}\vee1)\int_{(0,1)}\vert F_\nu^{-1}(u)-c\vert^{\rho-1}\vert F_\mu^{-1}(u)-F_\nu^{-1}(u)\vert\,du.
	\end{split}
	\end{align}
	
	Similarly, we have
	\begin{align}\label{majorationMrhorho5}\begin{split}
	&\int_{(0,1)^2}\vert F_\nu^{-1}(v)-F_\nu^{-1}(u)\vert^{\rho-1}(F_\mu^{-1}-F_\nu^{-1})^-(u)\,du\,\pi^Q(u,dv)\\
	&\le(2^{\rho-2}\vee1)\int_{(0,1)}\vert F_\nu^{-1}(u)-c\vert^{\rho-1}\vert F_\mu^{-1}(u)-F_\nu^{-1}(u)\vert\,du.
	\end{split}
	\end{align}
	
	Plugging \eqref{majorationMrhorho4} and \eqref{majorationMrhorho5} in \eqref{majorationMrhorho3} for the first inequality, using Hölder's inequality for the second inequality and the inverse transform sampling for the equality, we have
	\begin{align*}
	&\int_{(0,1)}\vert F_\nu^{-1}(v)-F_\nu^{-1}(u)\vert^{\rho-1}\vert F_\mu^{-1}(u)-F_\nu^{-1}(u)\vert\,du\,\pi^Q(u,dv)\\
	&\le2(2^{\rho-2}\vee1)\int_{(0,1)}\vert F_\nu^{-1}(u)-c\vert^{\rho-1}\vert F_\mu^{-1}(u)-F_\nu^{-1}(u)\vert\,du\\
	&\le2(2^{\rho-2}\vee1)\left(\int_{(0,1)}\vert F_\mu^{-1}(u)-F_\nu^{-1}(u)\vert^\rho\,du\right)^{1/\rho}\left(\int_{(0,1)}\vert F_\nu^{-1}(u)-c\vert^\rho\,du\right)^{(\rho-1)/\rho}\\
	&=2(2^{\rho-2}\vee1)\mathcal W_\rho(\mu,\nu)\left(\int_\R\vert y-c\vert^\rho\,\nu(dy)\right)^{(\rho-1)/\rho}.
	\end{align*}
	
	The latter inequality and \eqref{eq3} plugged in \eqref{majorationMrhorho2} then yields
	\[
	\mathcal M_\rho^\rho(\mu,\nu)\le(2^{\rho-1}\gamma_1+2(2^{\rho-2}\vee1)\gamma_2)\mathcal W_\rho(\mu,\nu)\left(\int_\R\vert y-c\vert^\rho\,\nu(dy)\right)^{(\rho-1)/\rho}.
	\]
	
	By taking in the right-hand side the infimum over all $(\gamma_1,\gamma_2)\in\R_+\times\R_+$ such that for all $x\in\R_+$, $\frac{x+x^\rho}{1+x}\le\gamma_1+\gamma_2x^{\rho-1}$ and over all $c\in\R$, we deduce that
	\[
	\mathcal M_\rho^\rho(\mu,\nu)\le K_\rho\mathcal W_\rho(\mu,\nu)\sigma_\rho^{\rho-1}(\nu).
	\]
	
	To complete the proof, it remains to prove \ref{it:encadrementCrho}. The definition \eqref{defCrho} of $C_\rho$ implies that $C_\rho\le K_\rho$.
	
	We have $\sup_{x\in\R_+}\frac{x+x^1}{1+x}=2$. Hence $K_1\le2$ and, by \eqref{minocrho}, $2\le C_1$. Therefore $C_1=K_1=2$.
	
	Let us next suppose that $\rho\ge 2$. Then $\R_+\ni x\mapsto (1+x)^\rho-x^\rho$ is a convex function above its tangent at the origin which writes $(1+x)^\rho-x^\rho\ge 1+\rho x$ so that $\frac{x+x^\rho}{1+x}\le (1+x)^{\rho-1}$ for all $x\in\R_+$ and $K_\rho\le 2^{\rho-1}\times 0+2(2^{\rho-2}\vee 1)\times 1=2^{\rho-1}$. Since the lower bound in \eqref{minocrho} is equal to $2^{\rho-1}$, we deduce that $C_\rho=K_\rho=2^{\rho-1}$.
	
	Let us finally suppose that $\rho\in(1,2)$.
	For all $x\in\R_+$, $\frac{x+x^\rho}{1+x}\le 1+x^{\rho-1}\le 1+(1+x)^{\rho-1}$, hence $K_\rho\le2^{\rho-1}+2$. Moreover, let $\gamma_2=\sup_{x\in(1,+\infty)}\frac{x+x^{\rho}}{(1+x)^\rho}$. By Remark \ref{rknewStabilityInequalityMajoration}, $\gamma_2>1$. Since $\R_+\ni x\mapsto (1+x)^\rho-x^\rho$ is non-decreasing, for $x\in[0,1]$, $(1+x)^\rho\ge x^\rho+1\ge x^\rho+x$. Hence $\forall x\in\R_+$, $\frac{x+x^\rho}{1+x}\le \gamma_2(1+x)^{\rho-1}$ and $K_\rho\le2\gamma_2=2\sup_{x\in(1,+\infty)}\frac{x+x^{\rho}}{(1+x)^\rho}$. We conclude that
	\[
	K_\rho\le\left(2^{\rho-1}+2\right)\wedge\left(2\sup_{x\in(1,+\infty)}\frac{x+x^{\rho}}{(1+x)^\rho}\right).
	\]
\end{proof}
\begin{rk} For $\rho=2$, by \eqref{expressionM2} for the first inequality and the fact that $\sigma_2(\nu)$ is the standard deviation of $\nu$, consequence of the bias-variance decomposition, for the last equality, we have
	\begin{align*}
	\mathcal W_2^2(\mu,\nu)&\le\mathcal M_2^2(\mu,\nu)\\
	&=\int_\R y^2\,\nu(dy)-\int_\R x^2\,\mu(dx)\\
	&\le\int_\R y^2\,\nu(dy)-\left(\int_\R x\,\mu(dx)\right)^2=\int_\R y^2\,\nu(dy)-\left(\int_\R y\,\nu(dy)\right)^2\\
	&=\sigma_2(\nu),
	\end{align*}
	where the inequalities are equalities as soon as $\mu$ is reduced to an atom. Therefore we can improve the constant $2^\rho$ in \eqref{majorationWrhorho} at least in the case $\rho=2$. We can then naturally wonder whether we can also improve this constant for any $\rho>1$. The constant
	\[
	C'_\rho=\sup_{\nu\in\mathcal P_\rho(\R)}\frac{\int_\R\left\vert y-\int_\R z\,\nu(dz)\right\vert^\rho\,\nu(dy)}{\int_\R\vert y\vert^\rho\,\nu(dy)}
	\]
	is studied in \cite{Mori}. For all $\nu\in\mathcal P_\rho(\R)$, let $c\in\R$ be such that $\sigma_\rho^\rho(\nu)=\int_\R\vert y-c\vert^\rho\,\nu(dy)$ and $\nu_c$ be the image of $\nu$ by $y\mapsto y-c$. Then we have
	\[
	\int_\R\vert z\vert^\rho\,\nu_c(dz)=\sigma_\rho^\rho(\nu),\quad\int_\R\left\vert y-\int_\R z\,\nu_c(dz)\right\vert^\rho\,\nu_c(dy)=\int_\R\left\vert y-\int_\R z\,\nu(dz)\right\vert^\rho\,\nu(dy)=\mathcal W_\rho^\rho(\mu_\nu,\nu),
	\]
	where we denote $\mu_\nu=\delta_{\int_\R y\,\nu(dy)}$, which is dominated by $\nu$ in the convex order. We deduce that
	\[
	C'_\rho=\sup_{\nu\in\mathcal P_\rho(\R)}\frac{\mathcal W_\rho^\rho(\mu_\nu,\nu)}{\sigma_\rho^\rho(\nu)}\cdot
	\]
	
	Yet by \cite[Theorem 2.3]{Mori} we have $C'_\rho\sim_{\rho\to+\infty}\frac{2^{\rho-1}}{\sqrt{2\mathrm e\rho}}$, which shows that we cannot lower the constant $2^\rho$ in \eqref{majorationWrhorho} by a factor more than $2\sqrt{2\mathrm e\rho}$ asymptotically for $\rho\to+\infty$.
\end{rk}

\section{On multidimensional generalisations}
\label{sec:Towards a multidimensional generalisation}
One may legitimately wonder whether the new martingale Wasserstein inequality \eqref{newStabilityInequality} holds in higher dimension $d\in\N^*$, that is if for all $\rho\ge1$, there exists $C\in\R_+^*$ such that for all $\mu,\nu\in\mathcal P_\rho(\R^d)$ satisfying $\mu\le_{cx}\nu$,
\begin{equation}\label{newStabilityInequalityGalDimension}
\mathcal M_\rho^\rho(\mu,\nu)\le C\mathcal W_\rho(\mu,\nu)\sigma_\rho^{\rho-1}(\nu).
\end{equation}

For all $d\in\N^*$ and $\rho\ge1$, we define $C_{\rho,d}$ by
\begin{equation}
C_{\rho,d}=\inf\left\{C>0\mid\forall\mu,\nu\in\mathcal P_\rho(\R^d)\mbox{ such that }\mu\le_{cx}\nu,\ \mathcal M_\rho^\rho(\mu,\nu)\le C\mathcal W_\rho(\mu,\nu)\sigma_\rho^{\rho-1}(\nu)\right\}.\label{defcrhod}
\end{equation}

The constant $C_{\rho,d}$ is well defined but is potentially infinite. Of course, for $d=1$, we get $C_{\rho,d}=C_\rho$. Moreover, $C_{\rho,d}$ depends a priori on the choice of the norm in $\R^d$, but since all norms on $\R^d$ are equivalent, $C_{\rho,d}$ is finite for one specific norm iff it is finite for any norm.
In the next subsection, we give lower-bounds of the constant $C_{\rho,d}$ depending on $\rho\ge 1$ but neither on $d\ge 2$ nor on the norm $\R^d$ is endowed with. In particular, by investigating the consequences for $\rho\ge 1$ of the very nice example introduced by Br\"uckerhoff and Juillet \cite{BJ} to show that $C_{1,d}=+\infty$ for $d\ge 2$, we extend this equality to $\rho\in\left[1,\frac{1+\sqrt{5}}{2}\right)$. Unfortunately, apart in the particular case $\rho=2$ already addressed in the introduction, we were not able to prove the finiteness of $C_{\rho,d}$ for some $\rho$ in the complement interval $\left[\frac{1+\sqrt{5}}{2},+\infty\right)$. That is the reason why, we investigated restricted classes of couples $(\mu,\nu)\in\mathcal P_\rho(\R^d)\times \mathcal P_\rho(\R^d)$ with  $\mu\le_{cx}\nu$ such that \eqref{newStabilityInequalityGalDimension} holds with a finite constant $C$ uniform over the class. In subsection \ref{secext1d}, we exhibit three classes such that $C$ is equal to the one-dimensional constant $C_\rho$. In subsection \ref{secscal}, we deal with the scaling case where $C=3\times 2^{\rho-1}$.
\subsection{Lower-bounds on the constant $C_{\rho,d}$}
\begin{prop2}\label{CdgreaterThan2} Let $d\in\N^*\backslash\{1\}$ and $\rho\ge1$. Regardless of the norm $\R^d$ is endowed with, we have
	\begin{align}
	&C_{\rho,d}=+\infty\mbox{ for }\rho\in\left[1,\frac{1+\sqrt{5}}{2}\right),\;C_{\frac{1+\sqrt{5}}{2},d}\ge 2^{\frac{\sqrt{5}-1}{2}}\left(\frac{3+\sqrt{5}}{2}\right)^{\frac{3-\sqrt{5}}{2}}\simeq 2.217\label{lbusc},\\&C_{\rho,d}\ge 2
	\mbox{ for }\rho\in \left(\frac{1+\sqrt{5}}{2},2\right)\mbox{ and }C_{\rho,d}\ge 2^{\rho-1}
	\mbox{ for }\rho\ge 2.
	\end{align}
\end{prop2}
The equality $C_{1,d}=+\infty$ for $d\ge 2$ was recently obtained by Br\"uckerhoff and Juillet \cite{BJ} by exhibiting the very nice Example \ref{exbj} in dimension $d=2$. We derive \eqref{lbusc} for $d=2$ by investigating the consequences of this example for $\rho\ge 1$.
The extension to $d>2$ then follows from the next lemma and the fact that $\R^2\ni(x_1,x_2)\mapsto |(x_1,x_2,0,\cdots,0)|$ is a norm whatever the norm $\vert\cdot\vert$ $\R^d$ is endowed with. The case $\rho\in\rho\in \left(\frac{1+\sqrt{5}}{2},2\right)$ is deduced in the same way from Example \ref{extri} in subsection \ref{secscal} while the case $\rho\ge 2$ follows from the next lemma combined with the equality $C_{\rho}=2^{\rho-1}$ established in Proposition \ref{newStabilityInequalityMajoration}. 
\begin{lemma}\label{CdIsIncreasing} Let $d,d'\in\N^*$ be such that $d'<d$, $\rho\ge1$, $\vert\cdot\vert$ be a norm on $\R^d$ and $\vert\cdot\vert'$ be a norm on $\R^{d'}$ satisfying the following consistency condition:
	\begin{equation}\label{consistencyConditionNorms}
	\exists\lambda>0,\;\forall x_1,\cdots,x_{d'}\in\R,\quad\vert(x_1,\cdots,x_{d'})\vert'=\lambda\vert(x_1,\cdots,x_{d'},0,\cdots,0)\vert.
	\end{equation}
	
	Then $C_{\rho,d'}\le C_{\rho,d}$ for $\R^d$ and $\R^{d'}$ respectively endowed with $\vert\cdot\vert$ and $\vert\cdot\vert'$.\\In particular, $C_\rho\le C_{\rho,d}$ regardless of the norm $\R^d$ is endowed with.
\end{lemma}
\begin{proof}[Proof of Lemma \ref{CdIsIncreasing}] Let $\mu',\nu'\in\mathcal P_\rho(\R^{d'})$ be such that $\mu'\le_{cx}\nu'$. Let $\mu$ and $\nu$ be the respective images of $\mu'$ and $\nu'$ by the map $\R^{d'}\ni(x_1,\cdots,x_{d'})\mapsto(x_1,\cdots,x_{d'},0,\cdots0)\in\R^{d}$. Let $c'=(c'_1,\cdots,c'_{d'})\in\R^{d'}$ and $c=(c'_1,\cdots,c'_{d'},0,\cdots,0)\in\R^{d}$. By \eqref{consistencyConditionNorms} and the definition of $C_{\rho,d}$, we have
	\begin{align*}
	\mathcal M_\rho^\rho(\mu',\nu')&=\lambda^\rho\mathcal M_\rho^\rho(\mu,\nu)\le C_{\rho,d}\lambda\mathcal W_\rho(\mu,\nu)\left(\lambda^\rho \int_{\R^{d}}\vert y-c\vert^\rho\,\nu(dy)\right)^{(\rho-1)/\rho}
	\\&=C_{\rho,d}\mathcal W_\rho(\mu',\nu')\left(\int_{\R^{d}}\vert y'-c'\vert'^\rho\,\nu'(dy')\right)^{(\rho-1)/\rho}.
	\end{align*}
	
	By taking the infimum over all $c'\in\R^{d'}$, we get $\mathcal M_\rho^\rho(\mu',\nu')\le C_{\rho,d}W_\rho(\mu',\nu')\sigma_\rho^{\rho-1}(\nu')$, hence $C_{\rho,d'}\le C_{\rho,d}$.
	
	In the particular case $d'=1$, since the absolute value on $\R$ is consistent with $\vert\cdot\vert$ for the coefficient $\lambda=\frac{1}{|(1,0,\cdots,0)|}$, we obtain that $ C_\rho\le C_{\rho,d}$.
\end{proof}\begin{Ex}[taken from \cite{BJ}]\label{exbj}
	Let $\mu_n=\frac{1}{n}\sum_{i=1}^n\delta_{(i,0)}$ and $\mu_nP_\theta(dy)=\int_{\R^2}P_\theta(x,dy)\mu_n(dx)$ for $\theta\in [0,\pi)$ where $P_\theta$ denotes the two dimensional Markov kernel defined by $$P_\theta(x,dy)=\frac{1}{2}\left(\delta_{x-(\cos\theta,\sin\theta)}(dy)+\delta_{x+(\cos\theta,\sin\theta)}(dy)\right).
	$$Since $\mu_n(dx)P_0(x,dy)P_\theta(x,dz)$ is a coupling between $\mu_nP_0(dy)$ and $\mu_nP_\theta(dz)$, one has $${\cal W}_\rho(\mu_nP_\theta,\mu_nP_0)\le|(1-\cos\theta,-\sin\theta)|\stackrel{\theta\to 0}{\longrightarrow} 0.$$
	As a consequence $\lim_{\theta\to 0}{\cal W}_\rho(\mu_n,\mu_nP_\theta)={\cal W}_\rho(\mu_n,\mu_nP_0)$ and $\lim_{\theta\to 0}\sigma_\rho(\mu_n P_\theta)=\sigma_\rho(\mu_n P_0)$. let us now compute those two limits. 
	Since $\mu_n$ and $\mu_nP_0=\frac{1}{2n}\left(\delta_{(0,0)}+\delta_{(1,0)}+\delta_{(n,0)}+\delta_{(n+1,0)}\right)+\frac 1n\sum_{i=2}^{n-1}\delta_{(i,0)}$ are both supported on the abscissa axis, using the comotonous coupling, one computes ${\cal W}_\rho(\mu_n,\mu_nP_0)=n^{-\frac 1\rho}|(1,0)|$.
	
	Let $c_\star=(\frac{n+1}{2},0)$. By invariance of $\mu_nP_\theta$ by $x\mapsto 2c_\star-x$ and convexity of the norm, for each $c\in\R^2$,
	\begin{align*}
	\int_{\R^2}&|x-c|^\rho\mu_nP_\theta(dx)=\int_{\R^2}\frac 12(|x-c|^\rho+|2c_\star-x-c|)\mu_nP_\theta(dx)\\&\ge\int_{\R^2}\frac 12(|x-c_\star|^\rho+|2c_\star-x-c_\star|)\mu_nP_\theta(dx)= \int_{\R^2}|x-c_\star|^\rho\mu_nP_\theta(dx).
	\end{align*}
	Hence $\sigma^\rho_\rho(\mu_nP_\theta)=\int_{\R^2}|x-c_\star|^\rho\mu_nP_\theta(dx)$ and $$\frac{\sigma^\rho_\rho(\mu_nP_0)}{|(1,0)|^\rho}=\frac{1}{2^\rho n}\bigg((n+1)^\rho+(n-1)^\rho+2\sum_{i=2}^{\lfloor \frac{n+1} 2\rfloor}(n+1-2i)^\rho\bigg)\sim_{n\to\infty}\frac{n^\rho}{2^{\rho}(\rho+1)}.$$
	According to Lemma 1.1 \cite{BJ}, for $\theta\in (0,\pi)$, $P_\theta$ is the only martingale coupling between $\mu_n$ and $\mu_nP_\theta(dy)$ so that ${\mathcal M}_\rho(\mu_n,\mu_nP_\theta)=|(\cos\theta,\sin\theta)|$ and $\lim_{\theta\to 0+}{\mathcal M}_\rho(\mu_n,\mu_nP_\theta)=|(1,0)|$. Note that this limit is not equal to ${\mathcal M}_\rho(\mu_n,\mu_nP_0)$ when $\rho<2$. Indeed, \begin{align*}
	\frac{1}{n}\sum_{i=2}^{n-1}\delta_{((i,0),(i,0))}+\frac{1}{2n}\bigg(&\frac{n}{n+1}\delta_{((1,0),(0,0))}+\delta_{((1,0),(1,0))}+\frac{1}{n+1}\delta_{((1,0),(n+1,0))}\\&+\frac{1}{n+1}\delta_{((n+1,0),(0,0))}+\delta_{((n,0),(n,0))}+\frac{n}{n+1}\delta_{((n,0),(n+1,0))}\bigg)
	\end{align*}
	is a martingale coupling between $\mu_n$ and $\mu_nP_0$ (in fact, one may easily check that its image by the projection on the first and third coordinates is the only martingale coupling between the first marginals of $\mu_n$ and $\mu_nP_0$ in the family $(M^Q)_{Q\in{\cal Q}}$) so that ${\mathcal M}_\rho(\mu_n,\mu_nP_0)\le \left(\frac{n^\rho+n}{n^2+n}\right)^{\frac 1\rho}
	|(1,0)|$ .
	
	Hence \begin{align*}
	\lim_{\theta\to 0+}\frac{{\mathcal M}^\rho_\rho(\mu_n,\mu_nP_\theta)}{{\mathcal W}_\rho(\mu_n,\mu_nP_\theta)\sigma_\rho^{\rho-1}(\mu_nP_\theta)}=\frac{|(1,0)|}{{\mathcal W}_\rho(\mu_n,\mu_nP_0)\sigma_\rho^{\rho-1}(\mu_nP_0)}\sim_{n\to\infty}2^{\rho-1}(\rho+1)^{\frac{\rho-1}{\rho}}n^{\frac{1}{\rho}-\rho+1}
	\end{align*}
	One has $\frac{1}{\rho}-\rho+1>0$ for $\rho\in\left[1,\frac{1+\sqrt{5}}{2}\right)$ and $\frac{1}{\rho}-\rho+1= 0$ for $\rho=\frac{1+\sqrt{5}}{2}$. Since $C_{\rho,2}\ge \lim_{n\to\infty}\lim_{\theta\to 0+}\frac{{\mathcal M}^\rho_\rho(\mu_n,\mu_nP_\theta)}{{\mathcal W}_\rho(\mu_n,\mu_nP_\theta)\sigma_\rho^{\rho-1}(\mu_nP_\theta)}$, we conclude that $C_{\rho,2}=+\infty$ when $\rho\in\left[1,\frac{1+\sqrt{5}}{2}\right)$ and $C_{\rho,2}\ge 2^{\rho-1}(\rho+1)^{\frac{\rho-1}{\rho}}$ when $\rho=\frac{1+\sqrt{5}}{2}$, whatever the norm $\R^2$ is endowed with.
\end{Ex}\subsection{Extensions of the one dimensional inequality}\label{secext1d}

We first look in Propositions \ref{proptens}, \ref{proprad} and \ref{CdplusPetitque2FonctionH} at extensions of the one dimensional inequality which give the same optimal constant. We begin with the fact that the martingale Wasserstein inequality \eqref{newStabilityInequality} can be tensorised: it holds in greater dimension when the marginals are independent.

\begin{prop2} \label{proptens}Let $d\in\N^*$ and $\mu_1,\nu_1\cdots,\mu_d,\nu_d\in\mathcal P_\rho(\R)$ be such that for all $1\le i\le d$, $\mu_i\le_{cx}\nu_i$. Let $\mu=\mu_1\otimes\cdots\otimes\mu_d$ and $\nu=\nu_1\otimes\cdots\otimes\nu_d$. Then $\mu\le_{cx}\nu$ and
	\begin{equation}\label{newStabilityInequalityInequality}
	\mathcal M_\rho^\rho(\mu,\nu)\le C_\rho \mathcal W_\rho(\mu,\nu)\sigma_\rho^{\rho-1}(\nu),
	\end{equation}
	where $\R^d$ is endowed with the $L^\rho$-norm.
\end{prop2}
\begin{proof} For all $1\le i\le d$, there exists by Lemma \ref{existenceCouplageMartingaleOptimal} below a martingale coupling $M_i\in\Pi^{\textrm M}(\mu_i,\nu_i)$ between $\mu_i$ and $\nu_i$, optimal for $\mathcal M_\rho(\mu_i,\nu_i)$. Let then $M$ be the probability measure on $\R^d\times\R^d$ defined by
	\[
	M(dx,dy)=\mu(dx)\,m_1(x_1,dy_1)\cdots\,m_d(x_d,dy_d)=M_1(dx_1,dy_1)\otimes\cdots\otimes M_d(dx_d,dy_d).
	\]
	
	It is clear that $M$ is a martingale coupling between $\mu$ and $\nu$, which shows that $\mu\le_{cx}\nu$, and
	\begin{align*}
	\mathcal M_\rho^\rho(\mu,\nu)&\le\int_{\R^d\times\R^d}\vert x-y\vert^\rho\,M(dx,dy)=\sum_{i=1}^d\int_{\R^d\times\R^d}\vert x_i-y_i\vert^\rho\,M(dx,dy)\\
	&=\sum_{i=1}^d\int_{\R\times\R}\vert x_i-y_i\vert^\rho\,M_i(dx_i,dy_i).
	\end{align*}
	
	Then for all $c=(c_1,\cdots,c_d)\in\R^d$ we have
	\begin{align}\label{tensorisation1}\begin{split}
	\mathcal M_\rho^\rho(\mu,\nu)&\le\sum_{i=1}^d\int_{\R\times\R}\vert x_i-y_i\vert^\rho\,M_i(dx_i,dy_i)=\sum_{i=1}^d\mathcal M_\rho^\rho(\mu_i,\nu_i)\\
	&\le C_\rho\sum_{i=1}^d\mathcal W_\rho(\mu_i,\nu_i)\left(\int_\R\vert y_i-c_i\vert^\rho\,\nu_i(dy_i)\right)^{(\rho-1)/\rho}\\
	&\le C_\rho\left(\sum_{i=1}^d\mathcal W_\rho^\rho(\mu_i,\nu_i)\right)^{1/\rho}\left(\sum_{i=1}^d\int_\R\vert y_i-c_i\vert^\rho\,\nu_i(dy_i)\right)^{(\rho-1)/\rho},
	\end{split}
	\end{align}
	where for the last inequality we applied Hölder's inequality to the sum over $i$. Let $P\in\Pi(\mu,\nu)$ be a coupling between $\mu$ and $\nu$. For $1\le i\le d$, let $P_i$ be the marginals of $P$ with respect to the coordinates $i$ and $i+d$, so that $P_i$ is a coupling between $\mu_i$ and $\nu_i$. Then
	\begin{align*}
	\sum_{i=1}^d\mathcal W_\rho^\rho(\mu_i,\nu_i)&\le\sum_{i=1}^d\int_{\R\times\R}\vert x_i-y_i\vert^\rho\,P_i(dx_i,dy_i)=\int_{\R^d\times\R^d}\sum_{i=1}^d\vert x_i-y_i\vert^\rho\,P(dx,dy)\\
	&=\int_{\R^d\times\R^d}\vert x-y\vert^\rho\,P(dx,dy).
	\end{align*}
	
	Since the inequality above is true for any coupling $P$ between $\mu$ and $\nu$, we get
	\begin{equation}\label{tensorisation2}
	\sum_{i=1}^d\mathcal W_\rho^\rho(\mu_i,\nu_i)\le\mathcal W_\rho^\rho(\mu,\nu),
	\end{equation}
	which is in fact even an equality according to \cite[Proposition 1.1]{jourdainAlfonsiSPL}. We then deduce from \eqref{tensorisation1} and \eqref{tensorisation2} that
	\begin{align*}
	\mathcal M_\rho^\rho(\mu,\nu)&\le C_\rho \mathcal W_\rho(\mu,\nu)\left(\sum_{i=1}^d\int_\R\vert y_i-c_i\vert^\rho\,\nu_i(dy_i)\right)^{(\rho-1)/\rho}\\
	&=C_\rho \mathcal W_\rho(\mu,\nu)\left(\int_{\R^d}\sum_{i=1}^d\vert y_i-c_i\vert^\rho\,\nu(dy)\right)^{(\rho-1)/\rho}\\
	&=C_\rho \mathcal W_\rho(\mu,\nu)\left(\int_{\R^d}\vert y-c\vert^\rho\,\nu(dy)\right)^{(\rho-1)/\rho}.
	\end{align*}
	
	By taking the infimum over all $c\in\R^d$, we get \eqref{newStabilityInequalityInequality}.
\end{proof}
We next turn to the case when, for some $\alpha\in\R^d$, the images of $\mu$ and $\nu$ by $$\R^d\ni x\mapsto\left(|x-\alpha|,1_{\{|x-\alpha|>0\}}\frac{x-\alpha}{|x-\alpha|}\right)$$ are product measures sharing the same second marginal. This in particular covers the case of radially symmetric measures $\mu$ and $\nu$. \begin{prop2}\label{proprad}
	Let $\R^d$ be endowed with any norm and $\mu,\nu\in\mathcal P_\rho(\R^d)$ be the respective images of $\bar{\mu}(dr)\eta(d\theta)$ and $\bar{\nu}(dr)\eta(d\theta)$ by $(r,\theta)\mapsto \alpha+r\theta$ where $\alpha\in\R^d$, $\bar{\mu},\bar\nu\in\mathcal P_\rho(\R)$ are such that  $\bar{\mu}(\R_+)=\bar\nu(\R_+)=1$ and $\eta\in\mathcal P(\R^d)$ is such that $\eta({\mathbb S}_{d-1})=1$ for ${\mathbb S}_{d-1}=\{\theta\in\R^d:|\theta|=1\}$ and $\eta$ is invariant by $x\mapsto -x$. If $\mu\le_{cx}\nu$, then ${\cal M}_\rho^\rho(\mu,\nu)\le C_\rho{\cal W}_\rho(\mu,\nu)\sigma_\rho^{\rho+1}(\nu)$.
\end{prop2}
\begin{proof}
	% We have that $\mu$ and $\nu$ are the respective images of $\bar{\mu}(dr)\eta(d\theta)$ and $\bar{\nu}(dr)\eta(d\theta)$ by $(r,\theta)\mapsto r\theta$ where $\eta$ is the uniform distribution on the unit sphere ${\mathbb S}_{d-1}=\{\theta\in\R^d:|\theta|=1\}$ and $\bar{\mu},\bar\nu\in\mathcal P_\rho(\R)$ are such that  $\bar{\mu}(\R_+)=\bar\nu(\R_+)=1$. 
	Since ${\cal M}_\rho$ and ${\cal W}_\rho$ (resp. $\sigma_\rho$) are (resp. is) preserved by taking the image of its two arguments (resp. its argument) by the translation vector $\alpha$, we suppose without loss of generality that $\alpha=0$.   Let $\tilde\mu(dt),\tilde\nu(du)\in\mathcal P_\rho(\R)$ be the respective images of $\bar\mu(dr)\frac{\delta_{-1}+\delta_{1}}{2}(ds)$ and $\bar\nu(dr)\frac{\delta_{-1}+\delta_{1}}{2}(ds)$ by $(r,s)\mapsto rs$.
	Let $\varphi:\R\to\R$ be a convex function. The function $\R_+\ni r\mapsto\frac{\varphi(-r)+\varphi(r)}{2}$ is non-decreasing and convex so that the function $f:\R^d\to\R$ defined by $f(x)=\frac{\varphi(-|x|)+\varphi(|x|)}{2}$ is convex. We have
	\begin{align*}
	\int_{\R^d}f(x)\mu(dx)=\int_{\R_+\times{\mathbb S}_{d-1}}\frac{\varphi(-|r\theta|)+\varphi(|r\theta|)}{2}\bar\mu(dr)\eta(d\theta)=\int_{\R}\varphi(t)\tilde\mu(dt)
	\end{align*}
	and, in the same way, $\int_{\R^d}f(y)\nu(dy)=\int_{\R}\varphi(u)\tilde\nu(du)$ so that $\mu\le_{cx}\nu$ implies that $\tilde\mu\le_{cx}\tilde\nu$. Let $\widetilde M\in\Pi^M(\tilde\mu,\tilde\nu)$ be optimal for ${\cal M}_\rho(\tilde\mu,\tilde\nu)$ and $M(dx,dy)$ denote the image of $\widetilde M(dt,du)\eta(d\theta)$ by $(t,u,\theta)\mapsto(t\theta,u\theta)$. The marginals of $M$ are the respective images of $\bar\mu(dr)\frac{\delta_{-1}+\delta_{1}}{2}(ds)\eta(d\theta)$ and $\bar\nu(dr)\frac{\delta_{-1}+\delta_{1}}{2}(ds)\eta(d\theta)$ by $(r,s,\theta)\mapsto rs\theta$. As the image of $\frac{\delta_{-1}+\delta_{1}}{2}(ds)\eta(d\theta)$ by $(s,\theta)\mapsto s\theta$  is equal to $\eta$, they are equal to $\mu(dx)$ and $\nu(dy)$. Moreover, for $\psi:\R^d\to\R$ measurable and bounded, using the martingale property of $\widetilde M$ for the second equality, we obtain
	\begin{align*}
	\int_{\R^d\times\R^d}\psi(x)y M(dx,dy)&=\int_{\R^d}\int_{\R\times\R}\psi(t\theta)u\theta \widetilde M(dt,du)\eta(d\theta)\\&=\int_{\R^d}\int_{\R}\psi(t\theta)t\theta\tilde\mu(dt)\eta(d\theta)=\int_{\R^d}\psi(x)x\mu(dx)
	\end{align*}
	so that $M\in\Pi^M(\mu,\nu)$. As a consequence,
	\begin{align}
	{\cal M}_\rho^\rho(\mu,\nu)&\le\int_{\R^d\times\R^d}|x-y|^\rho M(dx,dy)=\int_{\R\times\R\times{\cal S}_{d-1}}|t\theta-u\theta|^\rho\widetilde M(dt,du)\eta(d\theta)\notag\\&=\int_{\R\times\R}|t-u|^\rho\widetilde M(dt,du)={\cal M}_\rho^\rho(\tilde \mu,\tilde \nu)\le C_\rho{\cal W}_\rho(\tilde\mu,\tilde\nu)\sigma_\rho^{\rho-1}(\tilde\nu).\label{com1d}
	\end{align}
	For $y,c\in\R^d$, by the triangle inequality $|y|\le \frac{|y-c|+|y+c|}{2}$ so that by Jensen's inequality, $|y|^\rho\le \frac{|y-c|^\rho+|-y-c|^\rho}{2}$. Therefore for $(r,\theta)\in\R_+\times\R^d% {\cal S}_{d-1}
	$, $\frac{|r\theta|^\rho+|-r\theta|^\rho}{2}\le \frac{|r\theta-c|^\rho+|-r\theta-c|^\rho}{2}$ so that \begin{align*}
	\int_{\R^d}|y|^\rho\nu(dy)&=\int_{\R_+\times\{-1,1\}\times\R^d% {\cal S}_{d-1}
	}|rs\theta|^\rho\bar\nu(dr)\frac{\delta_{-1}+\delta_{1}}{2}(ds)\eta(d\theta)\\&\le \int_{\R_+\times\{-1,1\}\times\R^d% {\cal S}_{d-1}
	}|rs\theta-c|^\rho\bar\nu(dr)\frac{\delta_{-1}+\delta_{1}}{2}(ds)\eta(d\theta)=\int_{\R^d}|y-c|^\rho\nu(dy).
	\end{align*}
	As a consequence, \begin{equation}
	\sigma_\rho^{\rho}(\nu)=\int_{\R^d}|y|^\rho\nu(dy)=\int_{\R\times{\cal S}_{d-1}}|u\theta|^\rho \tilde\nu(du)\eta(d\theta)=\int_{\R}|u|^\rho\tilde\nu(du)\ge\sigma_\rho^{\rho}(\tilde\nu),\label{compmomcentr}
	\end{equation}
	where the last inequality is in fact an equality by a reasoning similar to the one which just lead to the first equality.
	
	The image of $P\in\Pi(\mu,\nu)$ optimal for ${\cal W}_\rho(\mu,\nu)$ by $\R^d\times\R^d\ni(x,y)\mapsto(|x|,|y|)$ belongs to $\Pi(\bar\mu,\bar\nu)$ so that, with the inequality $|x-y|\ge||x|-|y||$ deduced from the triangle inequality,
	$${\cal W}_\rho^\rho(\mu,\nu)=\int_{\R^d\times\R^d}|x-y|^\rho P(dx,dy)\ge \int_{\R^d\times\R^d}||x|-|y||^\rho P(dx,dy)\ge {\cal W}_\rho^\rho(\bar\mu,\bar\nu).$$
	On the other hand, for $\bar P\in\Pi(\bar\mu,\bar \nu)$ optimal for ${\cal W}_\rho(\bar\mu,\bar\nu)$, the image of $\bar P(dr,dv)\frac{\delta_{-1}+\delta_{1}}{2}(ds)$ by $(r,v,s)\mapsto (rs,vs)$ belongs to $\Pi(\tilde \mu,\tilde \nu)$ so that
	$${\cal W}_\rho^\rho(\bar\mu,\bar\nu)=\int_{\R\times\R\times\{-1,1\}}|rs-vs|^\rho \bar P(dr,dv)\frac{\delta_{-1}+\delta_{1}}{2}(ds)\ge {\cal W}_\rho^\rho(\tilde\mu,\tilde\nu).$$
	Hence ${\cal W}_\rho(\mu,\nu)\ge {\cal W}_\rho(\bar\mu,\bar\nu)\ge {\cal W}_\rho(\tilde\mu,\tilde\nu)$ where the first (resp. second) inequality can be proved to be an equality by an adaptation of the reasoning leading to the second (resp. first) one.
	Plugging ${\cal W}_\rho(\mu,\nu)\ge {\cal W}_\rho(\tilde\mu,\tilde\nu)$ together with \eqref{compmomcentr} into \eqref{com1d}, we conclude that  ${\cal M}_\rho^\rho(\mu,\nu)\le C_\rho{\cal W}_\rho(\mu,\nu)\sigma_\rho^{\rho-1}(\nu)$.
	
\end{proof}
We now look at two measures $\mu$ and $\nu$ such that for $X$ distributed according to $\mu$, there exists $\lambda\ge0$ such that $\nu$ is the probability distribution of $X+\lambda(X-\E[X])$ and the conditional probability distribution of $X$ given the direction of $X-\E[X]$ has mean $\E[X]$. In order to transcribe formally the latter condition, we give the following definition.

\begin{def2}\label{defDirectionDependent} Let $d\in\N^*\backslash\{1\}$ and $H:\R^d\to\R^d$ be a measurable map such that $H(\R^d)$ is a Borel subset of $\R^d$. We say that $H$ is \emph{direction-dependent} iff $\vert H(x)\vert=1$ for all $x\in\R^d$ and
	\[
	\forall x,y\in\R^d\backslash\{0\},\quad H(x)=H(y)\iff y\in\operatorname{Span}(x).
	\]
\end{def2}

In dimension $d\in\N^*\backslash\{1\}$, a natural example of a direction-dependent map $H:\R^d\to\R^d$ is given by the one defined for all $x=(x_1,\cdots,x_d)\in\R^d\backslash\{0\}$ by
\begin{equation}\label{defFonctionH}
\hspace{-4pt}\left\{
\begin{array}{rll}
H(x)&=\phantom{-}\frac{x}{\vert x\vert}&\textrm{if}\ x_1>0\textrm{ or there exists $i\in\{1,\cdots,d-1\}$ such that $x_1=\cdots=x_i=0$ and $x_{i+1}>0$};\\
H(x)&=-\frac{x}{\vert x\vert}& \textrm{otherwise},
\end{array}
\right.
\end{equation}
and $H(0)$ is any vector with norm $1$.

\begin{prop2}\label{CdplusPetitque2FonctionH} Let $d\in\N^*\backslash\{1\}$, $r\in[1,+\infty]$ and $\R^d$ be endowed with the $L^r$-norm. Let $\rho\ge1$, $\mu\in\mathcal P_\rho(\R^d)$ be with mean $\alpha\in\R^d$, $\lambda:\R^d\to\R_+$, $H:\R^d\to\R^d$ be a direction-dependent measurable map in the sense of Definition \ref{defDirectionDependent} and $\nu$ be the image of $\mu$ by the map $x\mapsto x+\lambda(H(x-\alpha))(x-\alpha)=\alpha+(1+\lambda(H(x-\alpha)))(x-\alpha)$.
	
	If $\E[\vert\lambda(H(X-\alpha))(X-\alpha)\vert^\rho]<+\infty$ and $\E[X\vert H(X-\alpha)]=\alpha$ almost surely for $X$ distributed according to $\mu$, then $\mu\le_{cx}\nu$. If moreover $\lambda$ is constant, then
	\begin{equation}\label{CdplusPetitque2FonctionHInequality}
	\mathcal M_\rho^\rho(\mu,\nu)\le C_\rho \mathcal W_\rho(\mu,\nu)\sigma_\rho^{\rho-1}(\nu).
	\end{equation}
\end{prop2}
%\begin{prop2}\label{CdplusPetitque2FonctionH} Let $d\in\N^*\backslash\{1\}$, $r\in[1,+\infty]$ and $\R^d$ be endowed with the $L^r$-norm. Let $\rho\ge1$, $\lambda\ge0$, $\mu\in\mathcal P_\rho(\R^d)$ be with mean $\alpha\in\R^d$ and $\nu$ be the image of $\mu$ by the map $x\mapsto x+\lambda(x-\alpha)=\alpha+(1+\lambda)(x-\alpha)$. If there exists a direction-dependent measurable map $H:\R^d\to\R^d$ in the sense of Definition \ref{defDirectionDependent} such that for $X$ distributed according to $\mu$, $\E[X\vert H(X-\alpha)]=\alpha$ almost surely, then
%	\begin{equation}\label{CdplusPetitque2FonctionHInequality}
%	\mathcal M_\rho^\rho(\mu,\nu)\le C_\rho \mathcal W_\rho(\mu,\nu)\sigma_\rho^{\rho-1}(\nu).
%	\end{equation}
%\end{prop2}
\begin{rk} Suppose that $\mu\in\mathcal P_\rho(\R^d)$ is symmetric with mean $\alpha\in\R^d$, that is $(x-\alpha)_\sharp\mu(dx)=(\alpha-x)_\sharp\mu(dx)$. Let $H$ be defined by \eqref{defFonctionH} and $X$ be distributed according to $\mu$. Then $(X-\alpha,H(X-\alpha))\overset{d}{=}(\alpha-X,H(\alpha-X))=(\alpha-X,H(X-\alpha))$, so $\E[X-\alpha\vert H(X-\alpha)]=\E[\alpha-X\vert H(X-\alpha)]$ a.s., hence $\E[X\vert H(X-\alpha)]=\alpha$ a.s.
\end{rk}

The proof of Proposition \ref{CdplusPetitque2FonctionH} relies on the following lemma, whose proof is deferred to Section \ref{sec:Lemmas}, which explains why we can endow $\R^d$ with the $L^r$-norm for $r\in[1,\infty]$. In the case $r=2$ of the Euclidean norm, the result is a simple property of the orthogonal projection.
\begin{lemma}\label{lemmaProjectionNorme} Let $d\in\N^*\backslash\{1\}$, $r\in[1,+\infty]$, $\R^d$ be endowed with the $L^r$-norm, $\mathbb S^{d-1}=\{a\in\R^d\mid\vert a\vert=1\}$ and $\operatorname{sgn}:\R\to\R,x\mapsto\1_{\{x\ge0\}}-\1_{\{x<0\}}$. For all $a=(a_1,\cdots,a_d)\in\mathbb S^{d-1}$ and $c=(c_1,\cdots,c_d)\in\R^d$, let $c_a$ be defined by
	\begin{equation}\label{defcindicea}
	c_a=\left\{
	\begin{array}{rcl}
	\left(\sum_{i=1}^dc_i\operatorname{sgn}(a_i)\vert a_i\vert^{r-1}\right)a&\text{ if }&r<+\infty\\
	c_i\operatorname{sgn}(a_i)a&\text{ if }&r=+\infty,\text{ where $i=\min\{j\in\{1,\cdots,d\}\mbox{ such that }\vert a_j\vert=1\}$}.
	\end{array}
	\right.
	\end{equation}
	
	Then
	\[
	\forall a\in\mathbb S^{d-1},\quad\forall c\in\R^d,\quad\forall y\in\operatorname{Span}(a),\quad\vert y-c_a\vert\le\vert y-c\vert.
	\]
\end{lemma}

\begin{proof}[Proof of Proposition \ref{CdplusPetitque2FonctionH}] Up to replacing $\mu$ and $\nu$ by their respective images by the map $x\mapsto x-\alpha$, we may suppose without loss of generality that $\alpha=0$.
	
	Let $(p(a,dx))_{a\in H(\R^d)}$ be a probability kernel such that $( H_\sharp\mu)(da)\,p(a,dx)$ is the image of $\mu$ by the map $x\mapsto( H(x),x)$. For all $a\in H(\R^d)$, let $\widetilde p(a,dy)$ be the image of $p(a,dx)$ by the map $x\mapsto(1+\lambda(H(x)))x$. For all $a\in H(\R^d)$, let $q(a,\cdot)$, resp. $\widetilde q(a,\cdot)$, be the image of $p(a,\cdot)$, resp. $\widetilde p(a,\cdot)$, by the map $y\mapsto\langle y,a\rangle$. We have
	\begin{align*}
	&\int_{H(\R^d)}\left(\int_{\R^d}\vert x\vert^\rho\,p(a,dx)\right)\,(H_\sharp\mu)(da)=\int_{\R^d}\vert x\vert^\rho\,\mu(dx)<+\infty,\\
	\textrm{and}\quad&\int_{H(\R^d)}\left(\int_{\R^d}\vert y\vert^\rho\,\tilde p(a,dy)\right)\,(H_\sharp\mu)(da)=\int_{\R^d}\vert(1+\lambda(H(x)))x\vert^\rho\,\mu(dx)<+\infty,
	\end{align*}
	so $H_\sharp\mu(da)$-almost everywhere, $p(a,\cdot)$ and $q(a,\cdot)$, and therefore $\widetilde p(a,\cdot)$ and $\widetilde q(a,\cdot)$, belong to $\mathcal P_\rho(\R^d)$. Moreover we see by definition of $H$ that for $H_\sharp\mu(da)$-almost all $a\in H(\R^d)$, $p(a,\operatorname{Span}(a))=1$, so
	\[
	p(a,dx)=\int_\R\delta_{ta}(dx)\,q(a,dt)\quad\text{and}\quad\widetilde p(a,dy)=\int_\R\delta_{sa}(dy)\,\widetilde q(a,ds).
	\]
	
	By assumption, we have $H_\sharp\mu(da)$-almost everywhere
	\[
	a\int_\R t\,q(a,dt)=\int_{\R^d}x\,p(a,dx)=0.\]
	
	Since $\widetilde q(a,\cdot)$ is the image of $q(a,\cdot)$ by the map $y\mapsto(1+\lambda(H(y)))y$, or equivalently by the map $y\mapsto(1+\lambda(a))y$, by Lemma \ref{scalingCouplageMartingale} below, for $H_\sharp\mu(da)$-almost all $a\in H(\R^d)$, $q(a,\cdot)\le_{cx}\widetilde q(a,\cdot)$. Up to replacing $q(a,\cdot)$ and therefore $\widetilde q(a,\cdot)$ by $\delta_0$ on a $H_\sharp\mu$-null set, we may suppose without loss of generality that $q(a,\cdot),\widetilde q(a,\cdot)\in\mathcal P_\rho(\R)$ and $q(a,\cdot)\le_{cx}\widetilde q(a,\cdot)$ for all $a\in H(\R^d)$. By \cite[Theorem 19.12]{Ch06} and since ${\cal P}_\rho(\R^d)$ is a closed subset of ${\cal P}(\R^d)$ endowed with the weak convergence topology, the map $H(\R^d)\ni a\mapsto (q(a,\cdot),\tilde q(a,\cdot))\in{\cal P}_\rho(\R)\times{\cal P}_\rho(\R)$ is measurable when the codomain is endowed with the product $\mathcal B\otimes\mathcal B$ of the Borel $\sigma$-algebra $\mathcal B$ on ${\cal P}_\rho(\R)$ associated with the weak convergence topology. With Lemma \ref{measCouplageMartingaleOptimal} below, we deduce that there exists a measurable map $H(\R^d)\ni a\mapsto M^a\in\mathcal P(\R\times\R)$  endowed with the $\sigma$-field generated by the weak convergence topology such that for each $a\in H(\R^d)$, $M^a$ belongs to $\Pi^{\textrm M}(q(a,\cdot),\widetilde q(a,\cdot))$ and is optimal for $\mathcal M_\rho(q(a,\cdot),\widetilde q(a,\cdot))$. Let $\widetilde M^a$ be the image of $M^a$ by the map $(t,s)\mapsto(ta,sa)$. Then the map $a\mapsto\widetilde M^a$ is also measurable, which is equivalent to say (see again \cite[Theorem 19.12]{Ch06}) that $(\widetilde M^a)_{a\in H(\R^d)}$ is a probability kernel from $H(\R^d)$ to $\R^d\times\R^d$. Hence we can define
	\[
	\overline M(dx,dy)=\int_{a\in H(\R^d)}\widetilde M^a(dx,dy)\,(H_\sharp\mu)(da).
	\]
	For all $a\in H(\R^d)$, $M^a$ is a martingale coupling between $q(a,\cdot)$ and $\widetilde q(a,\cdot)$, hence we easily see that $\widetilde M^a$ is a martingale coupling between the respective images of $q(a,\cdot)$ and $\widetilde q(a,\cdot)$ by the map $t\mapsto at$, namely $p(a,\cdot)$ and $\widetilde p(a,\cdot)$. Therefore one can also readily show that $\overline M$ is a martingale coupling between $\int_{a\in H(\R^d)}p(a,dx)\,(H_\sharp\mu)(da)=\mu(dx)$ and $\int_{a\in H(\R^d)}\widetilde p(a,dy)\,(H_\sharp\mu)(da)=\nu(dy)$.	Consequently, 
	\begin{align}\label{calculMrhorho1}\begin{split}
	\mathcal M_\rho^\rho&(\mu,\nu)\le\int_{\R^d\times\R^d}\vert y-x\vert^\rho\,\overline M(dx,dy)=\int_{H(\R^d)}\left(\int_{\R^d\times\R^d}\vert y-x\vert^\rho\,\widetilde M^a(dx,dy)\right)\,(H_\sharp\mu)(da)\\
	&=\int_{H(\R^d)}\left(\int_{\R\times\R}\vert s-t\vert^\rho\,M^a(dt,ds)\right)\,(H_\sharp\mu)(da)=\int_{H(\R^d)}\mathcal M_\rho^\rho(q(a,\cdot),\widetilde q(a,\cdot))\,(H_\sharp\mu)(da).
	\end{split}
	\end{align}
	
	Let $c\in\R^d$. For all $a\in H(\R^d)$, let $c_a$ be defined by \eqref{defcindicea} and $s_a\in\R$ be such that $c_a=s_aa$. If the map $\lambda$ is constant equal to some $\lambda\in\R_+$ (with a slight abuse of notation), then using the definition of $C_\rho$ for the first inequality, Lemma \ref{valeurW1scaling} below for the first equality, Lemma \ref{lemmaProjectionNorme} for the second inequality, Hölder's inequality for the third inequality and Lemma \ref{valeurW1scaling} again for the last equality (there the constancy of $\lambda$ plays a crucial role), we deduce that
	\begin{align}\label{calculMrhorho2}\begin{split}
	\mathcal M_\rho^\rho(\mu,\nu)&\le\int_{H(\R^d)}C_\rho\mathcal W_\rho(q(a,\cdot),\widetilde q(a,\cdot))\left(\int_\R\vert s-s_a\vert^\rho\,\widetilde q(a,ds)\right)^{(\rho-1)/\rho}\,(H_\sharp\mu)(da)\\
	&=C_\rho\lambda\int_{H(\R^d)}\left(\int_\R\vert t\vert^\rho\,q(a,dt)\right)^{1/\rho}\left(\int_\R\vert sa-s_aa\vert^\rho\,\widetilde q(a,ds)\right)^{(\rho-1)/\rho}\,(H_\sharp\mu)(da)\\
	&=C_\rho\lambda\int_{H(\R^d)}\left(\int_{\R^d}\vert x\vert^\rho\,p(a,dx)\right)^{1/\rho}\left(\int_{\R^d}\vert y-c_a\vert^\rho\,\widetilde p(a,dy)\right)^{(\rho-1)/\rho}\,(H_\sharp\mu)(da)\\
	&\le C_\rho\lambda\int_{H(\R^d)}\left(\int_{\R^d}\vert x\vert^\rho\,p(a,dx)\right)^{1/\rho}\left(\int_{\R^d}\vert y-c\vert^\rho\,\widetilde p(a,dy)\right)^{(\rho-1)/\rho}\,(H_\sharp\mu)(da)\\
	&\le C_\rho\lambda\left(\int_{H(\R^d)}\int_{\R^d}\vert x\vert^\rho\,p(a,dx)\,(H_\sharp\mu)(da)\right)^{1/\rho}\\
	&\phantom{\le\ }\times\left(\int_{H(\R^d)}\int_{\R^d}\vert y-c\vert^\rho\,\widetilde p(a,dy)\,(H_\sharp\mu)(da)\right)^{(\rho-1)/\rho}\\
	&\le C_\rho\lambda\left(\int_{\R^d}\vert x\vert^\rho\,\mu(dx)\right)^{1/\rho}\left(\int_{\R^d}\vert (1+\lambda)x-c\vert^\rho\,\mu(dx)\right)^{(\rho-1)/\rho}\\
	&=C_\rho\mathcal W_\rho(\mu,\nu)\left(\int_{\R^d}\vert y-c\vert^\rho\,\nu(dy)\right)^{(\rho-1)/\rho}.
	\end{split}
	\end{align}
	
	By taking the infimum over all $c\in\R^d$, we get \eqref{CdplusPetitque2FonctionHInequality}.

\end{proof}
% \begin{rk}
%   For $\rho\in\{1\}\cup[2,+\infty)$, we can still conclude by replacing $M^a$ by the inverse transform martingale coupling $M^a_{IT}$ between $q(a,\cdot)$ and $\tilde q(a,\cdot)$. Indeed, by Proposition \ref{newStabilityInequalityMajoration},  $$\int_{\R\times\R}\vert s-t\vert^\rho\,M^a_{IT}(dt,ds)\le C_\rho\mathcal W_\rho(q(a,\cdot),\widetilde q(a,\cdot))\left(\int_\R\vert s-s_a\vert^\rho\,\widetilde q(a,ds)\right)^{(\rho-1)/\rho}.$$
%   Moreover the continuity in ${\cal W}_1$ of this coupling with respect to its marginals stated in Proposition 5.10 \cite{JoMa18} together with the fact that convergence in ${\cal W}_\rho$ is equivalent to weak convergence and convergence of moments of order $\rho$, implies that it is also continuous in ${\cal W}_\rho$ and therefore measurable.
% \end{rk}

\subsection{The scaling case}\label{secscal}

We call scaling case the situation in which two measures $\mu$ and $\nu$ are such that for $X$ distributed according to $\mu$, there exists $\lambda\ge0$ such that $\nu$ is the probability distribution of $X+\lambda(X-\E[X])=\E[X]+(1+\lambda)(X-\E[X])$. In the previous section we already considered this case under an additional assumption on the conditional probability distribution of $X$, see Proposition \ref{CdplusPetitque2FonctionH}. We release here the latter constraint and study the impact on the constant $C$ in \eqref{newStabilityInequalityGalDimension}. 

\begin{prop2}\label{scalingConstant3} Let $d\in\N^*$ and $\R^d$ be endowed with any norm. Let $\rho\ge1$, $\lambda\ge0$ and $\mu\in\mathcal P_\rho(\R^d)$ be with mean $\alpha\in\R^d$. Let $\nu$ be the image of $\mu$ by the map $x\mapsto x+\lambda(x-\alpha)$. Then
	\begin{equation}\label{scalingConstant3Inequality}
	\mathcal M_\rho^\rho(\mu,\nu)\le2^{\rho-1}\frac{3+\lambda}{1+\lambda}\mathcal W_\rho(\mu,\nu)\sigma_\rho^{\rho-1}(\nu).
	\end{equation}
\end{prop2}
\begin{rk} Suppose that there exists a direction-dependent measurable map $H:\R^d\to\R^d$ in the sense of Definition \ref{defDirectionDependent} such that for $X$ distributed according to $\mu$, $\E[X\vert H(X-\alpha)]=\alpha$ almost surely. Then by Proposition \ref{CdplusPetitque2FonctionH}, we see that $2^{\rho-1}\frac{3+\lambda}{1+\lambda}$ could be replaced in \eqref{scalingConstant3Inequality} with $C_\rho$. In view of \eqref{encadrementCrho} and Remark \ref{rknewStabilityInequalityMajoration}, for $\rho\in(1,2)$, $C_\rho>2^{\rho-1}$ so $2^{\rho-1}\frac{3+\lambda}{1+\lambda}$ is sharper for $\lambda$ in a neighbourhood of $+\infty$. However, the smallest constant independent of $\lambda$ induced by \eqref{scalingConstant3Inequality} is $3\times2^{\rho-1}$, which is greater than $C_\rho$ by Proposition \ref{newStabilityInequalityMajoration} \ref{it:encadrementCrho} (using $2\le 2^\rho=2\times 2^{\rho-1}$ when $\rho\in(1,2)$).
\end{rk}
\begin{proof} For all $x\in\R^d$, let $m(x,dy)$ be the probability kernel defined by
	\begin{equation}
	m(x,dy)=\frac{1}{1+\lambda}\delta_{x+\lambda(x-\alpha)}(dy)+\frac{\lambda}{1+\lambda}\nu(dy).\label{defkerscal}
	\end{equation}
	For all measurable and bounded map $h:\R^d\to\R$, we have
	\begin{align*}
	\int_{\R^d\times\R^d}h(y)\,\mu(dx)\,m(x,dy)&=\frac{1}{1+\lambda}\int_{\R^d}h(x+\lambda(x-\alpha))\,\mu(dx)+\frac{\lambda}{1+\lambda}\int_{\R^d}h(y)\,\nu(dy)\\
	&=\int_{\R^d}h(y)\,\nu(dy).
	\end{align*}
	
	Moreover, for all $x\in\R^d$,
	\begin{align*}
	\int_{\R^d}y\,m(x,dy)&=\frac{1}{1+\lambda}(x+\lambda(x-\alpha))+\frac{\lambda}{1+\lambda}\int_{\R^d}(x'+\lambda(x'-\alpha))\,\mu(dx')\\
	&=\frac{1}{1+\lambda}(x+\lambda(x-\alpha))+\frac{\lambda}{1+\lambda}\alpha=x.
	\end{align*}
	
	So $\mu(dx)\,m(x,dy)$ is a martingale coupling between $\mu$ and $\nu$, and
	\begin{align*}
	\mathcal M_\rho^\rho(\mu,\nu)&\le\int_{\R\times\R}\vert y-x\vert^\rho\,\mu(dx)\,m(x,dy)\\
	&=\frac{1}{1+\lambda}\int_{\R^d}\lambda^\rho\vert x-\alpha\vert^\rho\,\mu(dx)+\frac{\lambda}{1+\lambda}\int_{\R^d\times\R^d}\vert y-x\vert^\rho\,\mu(dx)\,\nu(dy).
	\end{align*}
	
	On the one hand, using Lemma \ref{valeurW1scaling} below and the fact that $\mu(dx)\,\nu(dy)$ is a coupling between $\mu$ and $\nu$, we have
	\begin{align*}
	\int_{\R^d}\vert x-\alpha\vert^\rho\,\mu(dx)&=\frac{1}{\lambda^\rho}\mathcal W_\rho^\rho(\mu,\nu)=\frac{1}{\lambda^\rho}\mathcal W_\rho(\mu,\nu)\mathcal W_\rho^{\rho-1}(\mu,\nu)\\
	&\le\frac{1}{\lambda^\rho}\mathcal W_\rho(\mu,\nu)\left(\int_{\R^d\times\R^d}\vert y-x\vert^\rho\,\mu(dx)\,\nu(dy)\right)^{(\rho-1)/\rho}.
	\end{align*}
	
	On the other hand, Minkowski's inequality and Lemma \ref{valeurW1scaling} below yield
	\begin{align*}
	&\int_{\R^d\times\R^d}\vert y-x\vert^\rho\,\mu(dx)\,\nu(dy)\\
	&=\left(\int_{\R^d\times\R^d}\vert y-x\vert^\rho\,\mu(dx)\,\nu(dy)\right)^{\frac1\rho}\left(\int_{\R^d\times\R^d}\vert y-x\vert^\rho\,\mu(dx)\,\nu(dy)\right)^{\frac{\rho-1}{\rho}}\\
	&\le\left(\left(\int_{\R^d}\vert x-\alpha\vert^\rho\,\mu(dx)\right)^{\frac1\rho}+\left(\int_{\R^d}\vert y-\alpha\vert^\rho\,\nu(dy)\right)^{\frac1\rho}\right)\left(\int_{\R^d\times\R^d}\vert y-x\vert^\rho\,\mu(dx)\,\nu(dy)\right)^{\frac{\rho-1}{\rho}}\\
	&=(2+\lambda)\left(\int_{\R^d}\vert x-\alpha\vert^\rho\,\mu(dx)\right)^{\frac1\rho}\left(\int_{\R^d\times\R^d}\vert y-x\vert^\rho\,\mu(dx)\,\nu(dy)\right)^{{\frac{\rho-1}{\rho}}}\\
	&=\frac{2+\lambda}{\lambda}\mathcal W_\rho(\mu,\nu)\left(\int_{\R^d\times\R^d}\vert y-x\vert^\rho\,\mu(dx)\,\nu(dy)\right)^{{\frac{\rho-1}{\rho}}}.
	\end{align*}
	
	We deduce that
	\[
	\mathcal M_\rho^\rho(\mu,\nu)\le\frac{3+\lambda}{1+\lambda}\mathcal W_\rho(\mu,\nu)\left(\int_{\R^d\times\R^d}\vert y-x\vert^\rho\,\mu(dx)\,\nu(dy)\right)^{(\rho-1)/\rho}.
	\]
	
	Using Minkowski's inequality and the definition of convex order, for all $c\in\R$ we get
	\begin{align*}
	&\left(\int_{\R^d\times\R^d}\vert y-x\vert^\rho\,\mu(dx)\,\nu(dy)\right)^{(\rho-1)/\rho}\\
	&\le\left(\left(\int_{\R^d}\vert x-c\vert^\rho\,\mu(dx)\right)^{1/\rho}+\left(\int_{\R^d}\vert y-c\vert^\rho\,\nu(dy)\right)^{1/\rho}\right)^{\rho-1}\\
	&\le2^{\rho-1}\left(\int_{\R^d}\vert y-c\vert^\rho\,\nu(dy)\right)^{(\rho-1)/\rho}.
	\end{align*}
	
	By taking the infimum over all $c\in\R$, we get
	\[
	\mathcal M_\rho^\rho(\mu,\nu)\le2^{\rho-1}\frac{3+\lambda}{1+\lambda}\mathcal W_\rho(\mu,\nu)\sigma_\rho^{\rho-1}(\nu).
	\]
\end{proof}
As already seen in the previous proof, in the scaling case, by Lemma
\ref{valeurW1scaling} below, $\mathcal W^\rho_\rho(\mu,\nu)=\lambda^\rho\int_{\R^d}\vert x-\alpha\vert ^\rho\mu(dx)$. On the other hand, $$\sigma^\rho_\rho(\nu)\le \int_{\R^d}\vert y-\alpha\vert ^\rho\nu(dy)=(1+\lambda)^\rho\int_{\R^d}\vert x-\alpha\vert ^\rho\mu(dx)$$ so that
\begin{equation}
C_{\rho,d}\ge \frac{{\cal M}_\rho^\rho(\mu,\nu)}{\lambda(1+\lambda)^{\rho-1}\int_{\R^d}\vert x-\alpha\vert ^\rho\mu(dx)}.\label{minocro}
\end{equation} When, moreover, $\nu$ is supported on an affine basis of $\R^d$, there is a single martingale coupling between $\mu$ and $\nu$ (thus given by $\mu(dx)m(x,dy)$ with the kernel $m$ defined in \eqref{defkerscal}) and we are in a good position to derive a lower-bound for the constant $C_{\rho,d}$ defined in \eqref{defcrhod}. In the next example, we exploit this idea in dimension $d=2$.

\begin{Ex}\label{extri}
	Let $n\in\N^*$, $i=(0,0)$, $j=(1,0)$, $k=(\frac12,\frac{1}{n})$, $p=\frac{1}{2n}$, $q=\frac{1}{2n}$, $r=1-\frac1n$ and $\alpha=pi+qj+rk=(\frac 12,\frac 1 n-\frac 1{n^2})$. We set $\mu=p\delta_i+q\delta_j+r\delta_k$ and define $\nu$ as the image of $\mu$ by $x\mapsto x+\lambda (x-\alpha)$ with $\lambda>0$% , namely $\nu=p\delta_{i_\lambda}+q\delta_{j_\lambda}+r\delta_{k_\lambda}$ with $i_\lambda,j_\lambda$ and $k_\lambda$ denoting the respective images of $i,j$ and $k$ by this map
	.  By the definition \eqref{defkerscal} of the kernel $m$, we have 
	\begin{align*}
	{\cal M}^\rho_\rho(\mu,\nu)=\frac{\lambda^\rho}{1+\lambda}\int_{\R^2}\vert x-\alpha\vert ^\rho\mu(dx)+\frac{\lambda}{1+\lambda}\int_{\R^2\times\R^2}|x-y|^\rho\mu(dx)\nu(dy)\end{align*}
	where $\int_{\R^2\times\R^2}|x-y|^\rho\mu(dx)\nu(dy)$ goes to $\int_{\R^2\times\R^2}|x-y|^\rho\mu(dx)\mu(dy)$ as $\lambda \to 0+$. Taking the limit $\lambda\to 0+$ in \eqref{minocro}, we deduce that
	$$C_{\rho,2}\ge \1_{\{\rho=1\}}+\frac{\int_{\R^2\times\R^2}|x-y|^\rho\mu(dx)\mu(dy)}{\int_{\R^2}|x-\alpha|^\rho\mu(dx)}.$$
	As $n\to\infty$, whatever the norm $|\cdot|$ on $\R^2$, $\vert k-j\vert$, $\vert i-k\vert$, $\vert i-\alpha\vert$ and $\vert j-\alpha\vert$ converge to $\vert i-j\vert/2$ and $\vert k-\alpha\vert=\frac{1}{n^2}|(0,1)|$, which implies that
	\begin{align*}&\int_{\R^2}|x-\alpha|^\rho\mu(dx)=p|i-\alpha|^\rho+ q|j-\alpha|^\rho+ r|k-\alpha|^\rho\sim \frac{|i-j|^\rho}{2^\rho n}\\
	&\int_{\R^2\times\R^2}|x-y|^\rho\mu(dx)\mu(dy)=2\left(pq\vert j-i\vert^\rho+qr\vert k-j\vert^\rho+rp\vert i-k\vert^\rho\right)\sim \frac{|i-j|^\rho}{2^{\rho-1} n}.
	\end{align*}
	By taking the limit $n\to\infty$ in the last inequality, we conclude that $
	C_{\rho,2}\ge \1_{\{\rho=1\}}+2$, whatever the norm $\R^2$ is endowed with.  \end{Ex}

\section{Technical lemmas}
\label{sec:Lemmas}
This section is devoted to the statements and proofs of technical lemmas needed earlier in the paper.

\begin{lemma}\label{existenceCouplageMartingaleOptimal} Let $d\in\N^*$, $\rho\ge1$ and $\mu,\nu\in\mathcal P_\rho(\R^d)$ be such that $\mu\le_{cx}\nu$. Let $c:\R^d\times\R^d\to\R_+$ be lower semicontinuous. Then the infimum
	\[
	C(\mu,\nu)=\inf_{M\in\Pi^{\textrm M}(\mu,\nu)}\int_{\R^d\times\R^d}c(x,y)\,M(dx,dy)
	\]
	is attained, i.e. there exists $M\in\Pi^{\textrm M}(\mu,\nu)$ such that $C(\mu,\nu)=\int_{\R^d\times\R^d}c(x,y)\,M(dx,dy)$.
\end{lemma}
% \begin{rk} It is well known that the infimum is still attained when $(x,y)\mapsto\vert x-y\vert^\rho$ is replaced with a lower semicontinuous function $c:\R^d\times\R^d\to\R_+$.
% \end{rk}
\begin{proof} Let $M_n\in\Pi^{\textrm M}(\mu,\nu)$, $n\in\N$ be a sequence of martingale couplings between $\mu$ and $\nu$ such that
	\[
	\int_{\R^d\times\R^d}c(x,y)\,M_n(dx,dy)\underset{n\to+\infty}{\longrightarrow}C(\mu,\nu).
	\]
	
	The probability measures $\mu$ and $\nu$ are tight: for all $\varepsilon>0$ there exists a compact subset $K\subset\R^d$ such that $\mu(K)\ge1-\varepsilon$ and $\nu(K)\ge1-\varepsilon$. Therefore, for all $n\in\N$,
	\[
	M_n((K\times K)^\complement)\le M_n((K^\complement\times\R^d)\cup(\R^d\times K^\complement))\le\mu(K^\complement)+\nu(K^\complement)\le2\varepsilon.
	\]
	
	We deduce that $(M_n)_{n\in\N}$ is tight. By Prokhorov's theorem, there exists an increasing map $\varphi:\N\to\N$ such that $(M_{\varphi(n)})_{n\in\N}$ converges weakly towards $M\in\mathcal P(\R^d\times\R^d)$. Since the projections maps $(x,y)\mapsto x$ and $(x,y)\mapsto y$ are continuous, the respective marginals of $M_{\varphi(n)}$ converge to the respective marginals of $M$. Since for all $n\in\N$, $M_{\varphi(n)}$ has marginals $\mu$ and $\nu$, so does $M$, hence $M\in\Pi(\mu,\nu)$. Moreover, for all $n\in\N$ let $(X_n,Y_n)$ be a bivariate random variable distributed according to $M_n$ and $(X,Y)$ be distributed according to $M$. Since $\mu$ and $\nu$ belong to $\mathcal P_1(\R^d)$, $((X_n,Y_n))_{n\in\N}$ is uniformly integrable. Let $f:\R^d\to\R$ be a continuous and bounded map. Since $(X_{\varphi(n)},Y_{\varphi(n)})_{n\in\N}$ converges in distribution to $(X,Y)$, is uniformly integrable and $(x,y)\mapsto f(x)(y-x)$ is continuous with at most linear growth, we have
	\begin{align*}
	0&=\int_{\R^d\times\R^d}f(x)(y-x)\,M_{\varphi(n)}(dx,dy)=\E[f(X_{\varphi(n)})(Y_{\varphi(n)}-X_{\varphi(n)})]\\
	&\underset{n\to+\infty}{\longrightarrow}\E[f(X)(Y-X)]=\int_{\R^d\times\R^d}f(x)(y-x)\,M(dx,dy).
	\end{align*}
	
	We deduce that $M\in\Pi^{\textrm M}(\mu,\nu)$. Then by the Portmanteau theorem, we get
	\[
	C(\mu,\nu)\le\int_{\R^d\times\R^d}c(x,y)\,M(dx,dy)\le\liminf_{n\to+\infty}\int_{\R^d\times\R^d}c(x,y)\, M_{\varphi(n)}(dx,dy)=C(\mu,\nu),
	\]
	so $M$ is optimal for $C(\mu,\nu)$.
\end{proof}
\begin{lemma}\label{measCouplageMartingaleOptimal} Let $\rho\ge1$ and $\widetilde\Pi_\rho=\{(q,q')\in\mathcal P_\rho(\R)\times\mathcal P_\rho(\R)\mid q\le_{cx}q'\}$. There exists a measurable map $M_\star:\widetilde\Pi_\rho\to {\mathcal P}(\R\times\R)$ such that for each $(\mu,\nu)\in\widetilde\Pi_\rho$, $M_\star(\mu,\nu)\in{\rm Opt}_\rho(\mu,\nu):=\{M\in\Pi^{\textrm M}(\mu,\nu):{\cal M}_\rho^\rho(\mu,\nu)=\int_{\R\times\R}|x-y|^\rho\,M(dx,dy)\}$, where $\widetilde\Pi_\rho$ is endowed with the trace of $\mathcal B\otimes\mathcal B$ with $\mathcal B$ denoting the Borel $\sigma$-algebra of $\mathcal P_\rho(\R)$ endowed with the weak convergence topology and ${\mathcal P}(\R\times\R)$ with the Borel $\sigma$-algebra associated with the weak convergence topology.
\end{lemma}
The reason why this statement is retricted to $d=1$ is that the proof relies on the continuity of ${\cal M}_\rho$ with respect to its marginals which fails in higher dimension for $\rho\in[1,2)$ according to Example \ref{exbj} taken from \cite{BJ}.
\begin{proof}The reasoning is inspired from the proof of Corollary 5.22 \cite{Vi09}.
	The set $\widetilde\Pi_\rho$ is a closed subset of the Polish space $\mathcal P_\rho(\R)\times\mathcal P_\rho(\R)$ endowed with the product of the ${\cal W}_\rho$ topology. Therefore it is Polish.
	The set $\bigcup_{(q,q')\in\widetilde\Pi_\rho}\Pi^M(q,q')$ is a closed subset of the set ${\cal P}_\rho(\R\times\R)$ endowed with ${\cal W}_\rho$ where the map $M\mapsto \int_{\R\times\R}|x-y|^\rho M(dx,dy)$ is continuous. Since $\widetilde\Pi_\rho\ni (\mu,\nu)\mapsto {\cal M}_\rho^\rho(\mu,\nu)$ is continuous according to Corollary 1.2 \cite{BaPa19}, we deduce that the  set $\bigcup_{(q,q')\in\widetilde\Pi_\rho}{\rm Opt}_\rho(q,q')$ is Polish as a closed subset of the Polish space ${\cal P}_\rho(\R\times\R)$ endowed with ${\cal W}_\rho$. For each $(\mu,\nu)\in\widetilde\Pi_\rho$, ${\rm Opt}_\rho(\mu,\nu)$ is non-empty and compact for the ${\cal W}_\rho$ topology.  The map $$\bigcup_{(q,q')\in\widetilde\Pi_\rho}{\rm Opt}_\rho(q,q')\ni M(dx,dy)\mapsto (M(dx,\R^d),M(\R^d,dy))\in\widetilde\Pi_\rho$$ is onto and continuous (and therefore measurable). The measurable selection theorem implies that it admits a measurable right-inverse $\widetilde\Pi_\rho\ni(\mu,\nu)\mapsto M_\star(\mu,\nu)\in \bigcup_{(q,q')\in\widetilde\Pi_\rho}{\rm Opt}_\rho(q,q')$ such that for each $(\mu,\nu)\in\widetilde\Pi_\rho$, $M_\star(\mu,\nu)\in {\rm Opt}_\rho(\mu,\nu)$. Since ${\mathcal P}_\rho(\R\times\R)$ is a closed subset of ${\mathcal P}(\R\times\R)$ for the weak convergence topology and the Borel $\sigma$-algebras on ${\mathcal P}_\rho(\R\times\R)$ associated to the ${\cal W}_\rho$ and the weak convergence topologies coincide according to the next lemma, the map is still measurable if we consider ${\cal P}(\R\times\R)$ endowed with the Borel $\sigma$-algebra associated with the weak convergence topology as the codomain. The Borel $\sigma$-algebra on ${\mathcal P}_\rho(\R)$ endowed with ${\cal W}_\rho$ coincides with ${\mathcal B}$ according to the next lemma. Since ${\cal P}_\rho(\R)$ endowed with ${\cal W}_\rho$ is Polish, the Borel $\sigma$-algebra on the product space ${\cal P}_\rho(\R)\times {\cal P}_\rho(\R)$ thus coincides with ${\mathcal B}\otimes{\mathcal B}$, which concludes the proof.\end{proof}
\begin{lemma} \label{measurabilityWrho}Let $(E,d_E)$ be a Polish space, $\rho\ge1$ and $\mathcal P_\rho(E)$ be the set of probability measures on $E$ with finite $\rho$-th moment. Let $\mathcal B$ , resp. $\mathcal B_\rho$ be the Borel $\sigma$-algebra on $\mathcal P_\rho(E)$ with respect to the weak convergence topology, resp. the $\mathcal W_\rho$-distance topology. Then $\mathcal B=\mathcal B_\rho$.
\end{lemma}
\begin{proof} Since the $\mathcal W_\rho$-distance topology is finer than the weak convergence topology, we clearly have $\mathcal B\subset\mathcal B_\rho$. Therefore it remains to prove that $\mathcal B_\rho\subset\mathcal B$.
	
	Let $x_0\in E$ and $\Phi_\rho(E)$ be the set of all real-valued continuous functions $f$ on $E$ which satisfy the growth constraint
	\[
	\exists\alpha>0,\quad\forall x\in E,\quad\vert f(x)\vert\le\alpha(1+d_E^\rho(x,x_0)).
	\]
	
	For all $f\in\Phi_\rho(E)$, let $\widetilde f:\mathcal P_\rho(E)\to\R$ be the map defined for all $p\in\mathcal P_\rho(E)$ by $\widetilde f(p)=\int_Ef(x)\,p(dx)$. The $\mathcal W_\rho$-distance topology is then the weak topology on $\mathcal P_\rho(E)$ induced by the family $(\widetilde f)_{f\in\Phi_\rho(E)}$, that is the coarsest topology on $\mathcal P_\rho(E)$ for which $\widetilde f$ is continuous for all $f\in\Phi_\rho(E)$. Any open set for this topology is a union of finitely many intersections of sets of the form $\widetilde f^{-1}(U)$ where $f\in\Phi_\rho(E)$ and $U$ is an open subset of $\R$. On the one hand, $(\mathcal P_\rho(E),\mathcal W_\rho)$ is Polish \cite[Theorem 6.18]{Vi09} and therefore strongly Lindelöf, hence the latter union can be assumed at most countable. On the other hand, any open subset of $\R$ is an at most countable union of open intervals of $\R$. We deduce that any open set for the $\mathcal W_\rho$-distance topology is an at most countable union of finitely many intersections of at most countable unions of sets of the form $\widetilde f^{-1}((a,b))$ where $f\in\Phi_\rho(E)$ and $(a,b)\subset\R$. Since $\mathcal B$ is closed under countable unions and intersections, it suffices to show that every set of the form $\widetilde f^{-1}((a,b))$ belongs to $\mathcal B$ to conclude that any open set of the $\mathcal W_\rho$-distance topology belongs to $\mathcal B$ and therefore $\mathcal B_\rho\subset\mathcal B$.
	
	Let then $f\in\Phi_\rho(E)$ and $a,b\in\R$ be such that $a<b$ and let us show that $\widetilde f^{-1}((a,b))\in\mathcal B$, which will end the proof. For all $n\in\N$, let
	
	\[
	f_n:x\mapsto (f(x)\vee(-n))\wedge n,
	\]
	which is clearly continuous and bounded. Then for all $n\in\N$ and $p\in\mathcal P_\rho(E)$,
	\[
	\widetilde f_n(p)=\int_X ((f(x)\vee(-n))\wedge n)\,p(dx),
	\]
	which by the dominated convergence theorem converges to $\widetilde f(p)$ as $n\to+\infty$, hence
	\[
	\widetilde f^{-1}((a,b))=\bigcup_{k\in\N^*}\bigcup_{N\in\N}\bigcap_{n\ge N}\widetilde f_n^{-1}\left(\left(a+\frac1k,b-\frac1k\right)\right).
	\]
	
	Since the weak convergence topology is induced by the family of $\widetilde g$ for $g$ continuous and bounded, we have that $\widetilde f_n^{-1}((a,b))\in\mathcal B$ for all $n\in\N$, hence $\widetilde f^{-1}((a,b))\in\mathcal B$.
\end{proof}

\begin{proof}[Proof of Lemma \ref{lemmaProjectionNorme}] Let $a=(a_1,\cdots,a_d)\in\mathbb S^{d-1}$, $c=(c_1,\cdots,c_d)\in\R^d$, $y\in\operatorname{Span}(a)$ and $t\in\R$ be such that $y=t a$. Suppose first that $r=+\infty$. Then
	\begin{align*}
	\vert y-c_a\vert&=\vert t a-c_i\operatorname{sgn}(a_i)a\vert=\vert t-c_i\operatorname{sgn}(a_i)\vert=\vert t\vert a_i\vert-c_i\operatorname{sgn}(a_i)\vert=\vert(ta_i-c_i)\operatorname{sgn}(a_i)\vert\\
	&=\vert ta_i-c_i\vert\le\vert t a-c\vert=\vert y-c\vert.
	\end{align*}
	
	Suppose now that $r<+\infty$. Using the fact that $\vert a\vert=1$ for the second and third equalities, Hölder's inequality for the second inequality and the fact that $\vert\operatorname{sgn}(x)\vert=1$ for all $x\in\R$ for the last but one equality, we get
	\begin{align*}
	\vert y-c_a\vert&=\left\vert t a-\left(\sum_{i=1}^dc_i\operatorname{sgn}(a_i)\vert a_i\vert^{r-1}\right)a\right\vert\\
	&=\left\vert t\sum_{i=1}^d\vert a_i\vert^r-\sum_{i=1}^dc_i\operatorname{sgn}(a_i)\vert a_i\vert^{r-1}\right\vert\\
	&\le\sum_{i=1}^d\vert t\vert a_i\vert-c_i\operatorname{sgn}(a_i)\vert\vert a_i\vert^{r-1}\\
	&\le\left(\sum_{i=1}^d\vert t\vert a_i\vert-c_i\operatorname{sgn}(a_i)\vert^r\right)^{1/r}\left(\sum_{i=1}^d\vert a_i\vert^r\right)^{(r-1)/r}\\
	&=\left(\sum_{i=1}^d\vert(t a_i-c_i)\operatorname{sgn}(a_i)\vert^r\right)^{1/r}\\
	&=\vert t a-c\vert=\vert y-c\vert.
	\end{align*}
\end{proof}% \begin{rk} It is well known that the infimum is still attained when $(x,y)\mapsto\vert x-y\vert^\rho$ is replaced with a lower semicontinuous function $c:\R^d\times\R^d\to\R_+$.

\begin{lemma}\label{valeurW1scaling} Let $d\in\N^*$, $\R^d$ be endowed with any norm, $\rho\ge1$, $\lambda\ge0$, $\mu\in\mathcal P_\rho(\R^d)$ and $\alpha\in\R^d$. Let $\nu$ be the image of $\mu$ by the map $x\mapsto x+\lambda(x-\alpha)$. Then
	\begin{equation}\label{expressionW1ScalingConstantSpeedGeodesic}
	\mathcal W_\rho(\mu,\nu)=\lambda\left(\int_{\R^d}\vert x-\alpha\vert^\rho\,\mu(dx)\right)^{1/\rho}.
	\end{equation}
\end{lemma}
\begin{rk} \label{rkcsg} Let $\eta_0,\eta_1\in\mathcal P_\rho(\R^d)$ and $\gamma\in\Pi(\eta_0,\eta_1)$ be optimal for $\mathcal W_\rho(\eta_0,\eta_1)$. For all $t\in[0,1]$, let $\eta_t$ be the image of $\gamma$ by $(x,y)\mapsto(1-t)x+ty$. It is well known that the curve $[0,1]\ni t\mapsto\eta_t$ is a constant speed geodesic in $(\mathcal P_\rho(\R^d),\mathcal W_\rho)$ connecting $\eta_0$ to $\eta_1$ \cite[Theorem 7.2.2]{AmGiSa}. Moreover, for all $0\le s\le t\le1$, the image of $\gamma$ by $((1-s)x+sy,(1-t)x+ty)$ is an optimal transport plan between $\eta_s$ and $\eta_t$ for the $\mathcal W_\rho$-distance.
	
	In particular for $\eta_0=\delta_\alpha$ and $\eta_1=\nu$, the unique coupling $\gamma(dx,dy)=\delta_\alpha(dx)\,\nu(dy)$ is optimal for $\mathcal W_\rho(\eta_0,\eta_1)$, and for $t=1/(1+\lambda)$, $\eta_t=\mu$. Therefore, the image of $\gamma$ by $(x,y)\mapsto((1-t)x+ty,y)$, that is the image of $\mu$ by $x\mapsto(x,x+\lambda(x-\alpha))$, is an optimal transport plan between $\mu$ and $\nu$ for the $\mathcal W_\rho$-distance, which implies \eqref{expressionW1ScalingConstantSpeedGeodesic}.
\end{rk}
We add here a quick proof with the central elements of Remark \ref{rkcsg}.
\begin{proof}[Proof of Lemma \ref{valeurW1scaling}] We have, by the triangle inequality for the metric $\mathcal W_\rho$,
	\begin{align*}
	\left(\int_{\R^d}\vert y-\alpha\vert^\rho\,\nu(dy)\right)^{1/\rho}&=\mathcal W_\rho(\delta_\alpha,\nu)\le \mathcal W_\rho(\delta_\alpha,\mu)+\mathcal W_\rho(\mu,\nu)\\
	&=\left(\int_{\R^d}\vert x-\alpha\vert^\rho\,\mu(dx)\right)^{1/\rho}+\mathcal W_\rho(\mu,\nu),
	\end{align*}
	so
	\[
	\mathcal W_\rho(\mu,\nu)\ge\left(\int_{\R^d}\vert y-\alpha\vert^\rho\,\nu(dy)\right)^{1/\rho}-\left(\int_{\R^d}\vert x-\alpha\vert^\rho\,\mu(dx)\right)^{1/\rho}=\lambda\left(\int_{\R^d}\vert x-\alpha\vert^\rho\,\mu(dx)\right)^{1/\rho}.
	\]
	
	Since $\mu(dx)\,\delta_{x+\lambda(x-\alpha)}(dy)$ is a coupling between $\mu$ and $\nu$, we also have
	\[
	\mathcal W_\rho(\mu,\nu)\le\lambda\left(\int_{\R^d}\vert x-\alpha\vert^\rho\,\mu(dx)\right)^{1/\rho},
	\]
	hence $\mathcal W_\rho(\mu,\nu)=\lambda\left(\int_{\R^d}\vert x-\alpha\vert^\rho\,\mu(dx)\right)^{1/\rho}$.
\end{proof}

\begin{lemma}\label{scalingCouplageMartingale} Let $d\in\N^*$, $\lambda>0$, $\mu\in\mathcal P_1(\R^d)$ and $\alpha\in\R^d$. Let $\nu$ be the image of $\mu$ by the map $x\mapsto x+\lambda(x-\alpha)$. Then $\mu\le_{cx}\nu$ iff $\alpha$ is the mean of $\mu$.
\end{lemma}
\begin{proof} If $\mu\le_{cx}\nu$, then $\mu$ and $\nu$ have the same mean, so
	\[
	\int_{\R^d}x\,\mu(dx)=\int_{\R^d}y\,\nu(dy)=\int_{\R^d}x\,\mu(dx)+\lambda\left(\int_{\R^d}x\,\mu(dx)-\alpha\right),
	\]
	which implies that $\alpha=\int_{\R^d}x\,\mu(dx)$.
	
	Conversely, suppose that $\alpha=\int_{\R^d}x\,\mu(dx)$. Then $\alpha=\int_{\R^d}y\,\nu(dy)$ and for all convex function $f:\R^d\to\R$,
	\begin{align*}
	\int_{\R^d}f(x)\,\mu(dx)&=\int_{\R^d}f\left(\frac{\lambda}{1+\lambda}\alpha+\frac{1}{1+\lambda}(x+\lambda(x-\alpha))\right)\,\mu(dx)\\
	&\le\int_{\R^d}\left(\frac{\lambda}{1+\lambda}f(\alpha)+\frac{1}{1+\lambda}f(x+\lambda(x-\alpha))\right)\,\mu(dx)\\
	&=\frac{\lambda}{1+\lambda}f\left(\int_{\R^d}y\,\nu(dy)\right)+\frac{1}{1+\lambda}\int_{\R^d}f(y)\,\nu(dy)\\
	&\le\frac{\lambda}{1+\lambda}\int_{\R^d}f(y)\,\nu(dy)+\frac{1}{1+\lambda}\int_{\R^d}f(y)\,\nu(dy)\\
	&=\int_{\R^d}f(y)\,\nu(dy),
	\end{align*}
	where we used Jensen's inequality in the last inequality. We deduce that $\mu\le_{cx}\nu$.
\end{proof}

\section*{Acknowledgements}

We thank Nicolas Juillet for his remarks on a preliminary version of this paper and Nizar Touzi for his suggestion to investigate the martingale Wasserstein inequality for radial probability measures which lead to Proposition \ref{proprad}. We also thank the referees for their remarks that helped us to improve the paper.

	\bibliography{biblio}{}
	\bibliographystyle{abbrv}
\end{document}